\documentclass[12pt]{amsart}
\usepackage{amsmath}
\usepackage{amsthm}
\usepackage{amsfonts}
\usepackage{amssymb}
\newtheorem{Theorem}{Theorem}[section]
\newtheorem{Proposition}[Theorem]{Proposition}
\newtheorem{Lemma}[Theorem]{Lemma}
\newtheorem{Corollary}[Theorem]{Corollary}
\theoremstyle{definition}
\newtheorem{Definition}{Definition}[section]
\theoremstyle{remark}

\numberwithin{equation}{section}

\newcommand{\Z}{{\mathbb Z}}
\newcommand{\R}{{\mathbb R}}
\newcommand{\C}{{\mathbb C}}
\newcommand{\N}{{\mathbb N}}

\begin{document}

\title[Ergodic Jacobi matrices]{Ergodic Jacobi matrices and conformal maps}

\author{Injo Hur}

\address{Mathematics Department\\
University of Oklahoma\\
Norman, OK 73019}

\email{ihur@math.ou.edu}

\author{Christian Remling}

\address{Mathematics Department\\
University of Oklahoma\\
Norman, OK 73019}

\email{cremling@math.ou.edu}

\urladdr{www.math.ou.edu/$\sim$cremling}

\date{August 23, 2011}


\thanks{CR's work supported
by NSF grant DMS 0758594}
\begin{abstract}
We study structural properties of the Lyapunov exponent $\gamma$ and the density of states $k$
for ergodic (or just invariant) Jacobi matrices in a general framework. In this analysis,
a central role is played by the function $w=-\gamma+i\pi k$ as a conformal map between
certain domains. This idea goes back to Marchenko and Ostrovskii, who used this device in
their analysis of the periodic problem.
\end{abstract}
\maketitle
\section{Introduction and basic setup}
\subsection{Introduction}
In this paper, we present a general abstract analysis of the basic quantities that
are commonly used in the spectral theory of ergodic spaces of Jacobi matrices.
Our original inspiration came from the work of Marchenko-Ostrovskii on periodic
Schr\"odinger operators \cite{MO}, which is perhaps best known (definitely to us)
through the reinterpretation of
this material that was given in \cite{GT1,GT2}. Marchenko-Ostrovskii use certain conformal
maps to parametrize periodic
problems, and the same device can be used in a much more general setting. This is one
of the main themes of the present paper.

What we do here has some overlap with earlier work
on the direct and inverse spectral theory of ergodic and invariant Jacobi matrices,
most notably with the by now classical contributions of Kotani
\cite{CarKot,Kot85,Kot}. So some parts of this paper are expository in character. Rather than
focus exclusively on those parts that (we believe) are new, we have attempted
to give a unified, coherent presentation that starts almost from scratch.
In those parts where the results are not
new, we usually propose alternative arguments.
\subsection{Basic setup}
Let us now try to give a somewhat more detailed description of what we will do here.
Recall that a \textit{Jacobi matrix }is a difference operator on $u\in\ell^2$ of the form
\[
(Ju)_n = a_nu_{n+1}+a_{n-1}u_{n-1}+b_n u_n .
\]
Alternatively, one can represent $J$ by the following tridiagonal matrix with respect to
the standard basis of $\ell^2(\Z)$:
\[
J = \begin{pmatrix} \ddots & \ddots & \ddots &&&& \\ & a_{-2} & b_{-1} & a_{-1} &&&\\
&& a_{-1} & b_0 & a_0 && \\ &&& a_0 & b_1 & a_1 & \\ &&&& \ddots & \ddots & \ddots
\end{pmatrix}
\]
Here, $a_n\ge 0$ and $b_n\in\R$, and we also assume that $a,b\in\ell^{\infty}(\Z)$.
Under these assumptions, $J$ is a bounded self-adjoint operator on $\ell^2(\Z)$.
(One often insists that $a_n>0$, but for what we want to do here, our convention works better.)

We will also impose a uniform bound on the operator norm,
and we will in fact work with specifically the space $\mathcal J_2$
of all such Jacobi matrices $J$ that satisfy
$\|J\|\le 2$; an arbitrary bounded Jacobi matrix will of course lie in $\mathcal J_2$ after
multiplication by a suitable constant. It is often useful to make $\mathcal J_2$
a compact metric space; one possible choice for such a metric is
\begin{equation}
\label{defd}
d(J,J') = \sum_{n\in\Z} 2^{-|n|} \left( |a_n-a'_n| + |b_n-b'_n| \right) .
\end{equation}
The topology induced by $d$ may be described as the product topology on $\mathcal J_2$,
now thought of as a subspace of the product of the intervals $[0,2]$ and $[-2,2]$ from
which we draw the coefficients $a_n$ and $b_n$, respectively. Alternatively,
this topology is also the one induced by both the weak and the strong operator topologies,
and we now think of $\mathcal J_2$ as a subspace of $B(\ell^2)$, the bounded operators
on the Hilbert space $\ell^2(\Z)$.

The shift $S(a,b)_n=(a,b)_{n+1}$ acts as a homeomorphism on $(\mathcal J_2, d)$.
Given an $S$ invariant probability (Borel) measure $\mu$ on $\mathcal J_2$,
we introduce a $w$ function $w=w_{\mu}$ as follows. We average the
spectral measures $d\rho_0(t;J)=d\|E_J(t)\delta_0\|^2$ with respect to $\mu$
to obtain the \textit{density of states }measure $dk$: More precisely, the
map $f\mapsto \int d\mu(J) \int d\rho_0(t;J) f(t)$ defines a positive linear
functional on the continuous functions $f$ on $[-2,2]$, so there exists a unique
(probability) measure $dk$ on the Borel sets of $[-2,2]$ so that
\begin{equation}
\label{defDOS}
\int_{\mathcal J_2} d\mu(J) \int_{[-2,2]} d\rho_0(t;J) f(t) = \int_{[-2,2]} f(t)\, dk(t)
\end{equation}
for all $f\in C[-2,2]$. It's easy to see that $J\mapsto \int f(t)\, d\rho_0(t;J)$ is a continuous
map on $\mathcal J_2$ for fixed $f\in C[-2,2]$; we will discuss this in more detail
in the proof Lemma \ref{Ldkcont} below. In particular, this function is measurable
and thus the left-hand side of \eqref{defDOS} is well defined.

We also define $A>0$ by writing
\[
\int_{\mathcal J_2} \ln a_0(J)\, d\mu(J) = \ln A ,
\]
at least if $\int \ln a_0\, d\mu>-\infty$. For easier reference,
we introduce the notation $\mathcal M_0$ for the set of ($S$ invariant, probability) Borel measures
$\mu$ on $\mathcal J_2$ that satisfy this additional condition. We then set
\begin{equation}
\label{defw}
w(z) = \ln A - \int_{[-2,2]} \ln(t-z)\, dk(t) ,
\end{equation}
for $z\in\mathbb C^+$, the upper half plane in $\C$. Here we take the logarithm
with $\textrm{Im }\ln\zeta\in (-\pi,0)$ for $\zeta\in\C^-$. So in particular
$w$ is a Herglotz function (a holomorphic function $w:\C^+\to\C^+$).
The harmonic (on $\C^+$) function
$\gamma(z)=-\textrm{Re }w(z)$ is called the \textit{Lyapunov exponent.}

These are, of course, well known quantities for ergodic systems of Jacobi
matrices, extended here in an obvious way to measures $\mu$ that are just invariant.
These quantities are often defined in different ways, and indeed
there are quite
a few well known alternative methods to introduce $w$. See \cite{CarLa,PasFig,Teschl}
for (much) more on these topics.
Definition \eqref{defw} is straightforward and convenient for our purposes.
\subsection{Overview of main themes}
As already announced, one of the recurring themes of this paper will be the generalized version of
the observation of Marchenko and Ostrovskii that $w$ maps $\C^+$ conformally
onto an image domain $w(\C^+)$ of a certain type, and, conversely,
these domains can be used to reconstruct
$w$, $A$, and $dk$ (in fact, this is not literally true; it becomes true after a suitable
change of variables, as we'll discuss below). See the discussion of Section 2,
especially Proposition \ref{P2.2}.
A variety of other data are available,
and we study the relations between these in some detail in Sections 3 and 4.
Section 5 contains one of the main results of this paper, Theorem \ref{T5.1}:
Given suitable data (for example, given a $w$ function), we can find an invariant
measure $\mu\in\mathcal M_0$ that will produce these data. A less complete
result of this type was proved earlier in \cite{CarKot}.

A new topic is introduced in Section 6. Here we show that the correspondence between
the gaps of the spectrum and slits of the image domain that is one of the cornerstones
of the Marchenko-Ostrovskii method (and obvious in the original setting) extends
to the general case, if suitable definitions are made. In Section 7, we study
the Lyapunov exponent as a function on $[-2,2]$ (rather than as a harmonic function
on $\C^+$).

Sections 2--7 form the main part of this paper. The final three sections are
lighter in tone. In Section 8, we revisit work of Avila-Damanik \cite{AD} on the positivity
of generic Lyapunov exponents from the point of view suggested by the material
of this paper. Section 9 offers a brief discussion of the possibility of finding
\textit{ergodic }(and not just invariant) measures $\mu$, but what we have to say
here does not really go beyond the work of Kotani \cite{Kot85}, and we present
more questions than answers. In the final section, we give an easy argument
for the invariance of $w(z)$ under a class of transformations that includes
all Toda flows.
\section{Basic objects}
Given $\mu\in\mathcal M_0$, define the corresponding $w$ as described above. Write
\[
w(z) = -\gamma(z) + i\pi k_0(z) ;
\]
notice that $0<k_0<1$. Also, let
\[
k_1(t)=\int_{(-\infty,t]} dk(s)
\]
be the increasing function that generates the density of states measure $dk$.
\begin{Proposition}
\label{P2.1}
(a) Let
\[
k(z) = \begin{cases} k_0(z) & z\in\C^+ \\ k_1(z) & z\in\R \end{cases} .
\]
Then $k$ is continuous on $\C^+\cup\R$.\\
(b) The limit
\[
\gamma(x):=\lim_{y\to 0+} \gamma(x+iy)
\]
exists for all $x\in\R$. Moreover, $\gamma(z)>0$ on $z\in\C^+$.\\
(c) (\textit{Thouless formula}) For all $z\in\C^+\cup\R$,
\[
\gamma(z) = -\ln A + \int_{[-2,2]} \ln |t-z|\, dk(t)
\]
\end{Proposition}
These properties are well known for \textit{ergodic }measures $\mu\in\mathcal M_0$. See, for example,
\cite[Chapter 5]{Teschl}.
A discussion of these issues for measures $\mu$ that are just invariant
may be found in \cite{CarKot}.
\begin{proof}[Sketch of proof]
Perhaps the most interesting part of this proof is the one where we establish the inequality $\gamma>0$
on $\C^+$; once this is available,
everything else will then fall into place very quickly or at least follow
from routine arguments. Let us first sketch how this can be done,
assuming, for the moment, the inequality $\gamma>0$.

Indeed, part (c) for $z\in\C^+$ is of course an immediate consequence of the
definitions of $w$ and $\gamma$. Existence of the limit from part (b) can then be deduced
from (c) by splitting the region of integration into the two parts $|t-x|\le 1$
and $|t-x|>1$ and using monotone and dominated convergence, respectively.
These considerations also extend the validity of (c) to $z\in\R$.

Next, we observe that the inequality $\gamma>0$ together with the Thouless formula
force $dk$ to be continuous measure; equivalently, $k_1(t)$ is a continuous function
on $\R$.

Define
\begin{equation}
\label{0.3}
k_0(t)=\lim_{y\to 0+} k_0(t+iy) ;
\end{equation}
the limit exists for (Lebesgue)
almost every $t\in\R$. Since $k_0(z)$ is bounded, the Herglotz representation of $w(z)$ reads
\begin{equation}
\label{2.11}
w(z) = C_0+Dz + \int_{-\infty}^{\infty} \left( \frac{1}{t-z}-\frac{t}{t^2+1} \right) k_0(t)\, dt .
\end{equation}
In fact, as $\textrm{Im }w(z)<\pi$ on $\C^+$, we must have $D=0$ here.
By differentiating \eqref{defw}, we obtain that $w'(z)=\int\frac{dk(t)}{t-z}$, so $\textrm{Im }w'(z)>0$,
or, equivalently, $\partial k_0(x+iy)/\partial x>0$ on $\C^+$. This implies that
$k_0(t)$ is an increasing function on $\R$. Originally, we could only guarantee that
$k_0(t)$ was defined off a null set $N\subset\R$,
but now we can put $k_0(s)=\lim_{t\to s-;t\notin N}k_0(t)$ for
$s\in N$ to obtain an everywhere (on $\R$) defined increasing function $k_0$.
It is also clear, by direct inspection of \eqref{defw}, that $k_0(t)=0$ for $t<-2$
and $k_0(t)=1$ for $t>2$. Thus $k_0$ generates a probability measure $dk_0$ on $[-2,2]$,
and now an integration by parts lets us rewrite \eqref{2.11} as follows:
\begin{align}
\label{wintparts}
w(z) & = C_0 + \int_{-\infty}^{\infty} \frac{\partial}{\partial t}
\left[ \ln (t-z) - \frac{1}{2}\ln(t^2 + 1) \right] k_0(t)\, dt \\
\nonumber
& = C_0 + \lim_{R\to\infty} k_0(t)\ln\frac{t-z}{\sqrt{t^2+1}} \Bigr|_{t=-R}^{t=R} \\
\nonumber
& \quad\quad
-\int_{-\infty}^{\infty} \left[ \ln (t-z) - \frac{1}{2}\ln(t^2 + 1) \right]\, dk_0(t)\\
\nonumber
& = C - \int_{-\infty}^{\infty} \ln (t-z) \, dk_0(t) .
\end{align}
Measures in Herglotz representations are unique and we can again consider $w'$, so
it follows from this that $dk_0=dk_1$. As already observed
above, $k_1$ is a continuous function on $\R$, and hence so is $k_0(t)=k_1(t)$.
Moreover, we defined $k_0(t)$, in \eqref{0.3}, as the boundary value,
Lebesgue almost everywhere,
of the bounded harmonic function $k_0(z)$, $z\in\C^+$. The Poisson representation formula
now shows that $k=k_0$ is continuous on $\C^+\cup\R$, as claimed.

So, as promised, it only remains to show that $\gamma>0$. We will in fact assume
this inequality for ergodic $\mu$. This is well known; in the ergodic case,
$\gamma$ can be related to the exponential decay rate of certain solutions to
the difference equation $Ju=zu$ (thus the term \textit{Lyapunov exponent}).
See \cite[Chapter 5]{Teschl}. So we will only explain how
to generalize the inequality to invariant $\mu$. As mentioned above, this issue is also discussed
in \cite{CarKot}; we offer an easy alternative argument here.

Let
\begin{equation}
\label{defFe}
F_{\epsilon}(J) = \frac{1}{1+\epsilon} (J+\epsilon J_0) ,
\end{equation}
where $J_0$ is the Jacobi matrix with $a_n\equiv 1$, $b_n\equiv 0$. In other words,
we essentially add $\epsilon$ to all $a$'s; the denominator $1+\epsilon$ is not essential
and is only introduced to make sure that $F_{\epsilon}(J)\in\mathcal J_2$ again.
Given an invariant measure $\mu$, let $\mu_{\epsilon}=F_{\epsilon}\mu$ be the corresponding
image measure; in other words, $\int f\, d\mu_{\epsilon} = \int f\circ F_{\epsilon}\, d\mu$.

Then $\mu_{\epsilon}$ is an invariant measure on the compact subspace
\[
\mathcal J_2^{(\epsilon )}= \{ J\in\mathcal J_2 : a_n\ge \frac{\epsilon}{1+\epsilon}
\textrm{ for all }n\in\Z \}
\]
of $\mathcal J_2$.
Since the ergodic measures are the extreme points of the set of invariant measures, there
are convex combinations $\mu_{\epsilon}^{(n)}$ of \textit{ergodic }measures $\nu_{j,n,\epsilon}$
on $\mathcal J_2^{(\epsilon)}$,
\[
\mu_{\epsilon}^{(n)} = \sum_{j=1}^{N_{n,\epsilon}} c_{j,n,\epsilon} \nu_{j,n,\epsilon} ,
\]
so that $\mu_{\epsilon}^{(n)}\to\mu_{\epsilon}$ in weak-$*$ sense as $n\to\infty$.
By the result for ergodic measures, we do have that $\gamma_{j,n,\epsilon}>0$ for the
corresponding Lyapunov exponents, and since $\gamma_{\nu}$ depends linearly on $\nu$,
it also follows that $\gamma_{\epsilon}^{(n)}>0$. Now on $\mathcal J_2^{(\epsilon )}$,
the function $J\mapsto \ln a_0(J)$ is continuous, so $\ln A_{\epsilon}^{(n)}\to
\ln A_{\epsilon}$ as $n\to\infty$.

The integrals from the Thouless formula
will also converge. To see this, we make use of the following simple fact.
\begin{Lemma}
\label{Ldkcont}
Suppose that $\mu_n\to\mu$ in weak-$*$ sense. Then also $dk_n\to dk$.
\end{Lemma}
The situation we have in mind here of course includes the assumption that
$\mu_n,\mu\in\mathcal M_0$, but the Lemma is also valid, with the same proof,
for arbitrary finite measures.
\begin{proof}[Proof of Lemma \ref{Ldkcont}]
Let $f\in C[-2,2]$. Then, from \eqref{defDOS},
\begin{align*}
\int f(t)\, dk(t) & = \int_{\mathcal J_2} d\mu(J)\int_{[-2,2]} d\rho_0(t;J)\, f(t) \\
& = \lim_{n\to\infty} \int_{\mathcal J_2} d\mu_n(J)\int_{[-2,2]} d\rho_0(t;J)\, f(t)\\
& = \lim_{n\to\infty} \int f(t)\, dk_n(t)
\end{align*}
because $J\mapsto \int f(t)\, d\rho_0(t;J)$ is a continuous function on $\mathcal J_2$.
To confirm this last claim,
it suffices to observe that convergence with respect to $d$ is equivalent
to strong operator convergence and this, in turn, implies weak-$*$ convergence
of the spectral measures $\rho_0$.
\end{proof}

Thus we now know that $\gamma_{\epsilon}^{(n)}(z)\to \gamma_{\epsilon}(z)$ on
$z\in\C^+$. In particular, it follows that $\gamma_{\epsilon}\ge 0$ there.

From the definition of $\mu_{\epsilon}$ and dominated convergence, it is also clear
that $\mu_{\epsilon}\to\mu$ in weak-$*$ sense as $\epsilon\to 0+$. Hence, as just observed,
the integrals from the Thouless formula approach the corresponding limit as $\epsilon\to 0+$.
Finally, monotone convergence shows that
\begin{align*}
\ln A_{\epsilon} & = \int\ln a_0(J)\, d\mu_{\epsilon}(J)
= \int\ln a_0(F_{\epsilon}(J))\, d\mu(J)\\
& = \int \ln \frac{a_0(J)+\epsilon}{1+\epsilon}\, d\mu(J) \to \int\ln a_0(J)\, d\mu(J) = \ln A .
\end{align*}
Hence also $\gamma_{\epsilon}\to \gamma$, so $\gamma\ge 0$.
The harmonic function
$\gamma$ is clearly not equal to a constant, hence cannot assume
a minimum value, and thus in fact $\gamma>0$ on $\C^+$.
\end{proof}
It is also useful to notice the following well known consequence of basic
potential theory at an early stage:
\begin{Lemma}
\label{L2.5}
$A\le 1$ for any $\mu\in\mathcal M_0$.
\end{Lemma}
\begin{proof}
Integrate the Thouless formula with respect to $dk$. Since $\gamma\ge 0$, we obtain that
\[
0 \le -\ln A + \int\!\!\!\int \ln |t-x|\, dk(t)\,dk(x) .
\]
By the definition of logarithmic capacity \cite[Definition 5.1.1]{Ran}, the double
integral is $\le\ln\textrm{cap }[-2,2]= 0$.
\end{proof}
This argument also shows that if $A=1$, then $dk=d\omega_{[-2,2]}$,
the equilibrium measure of $[-2,2]$. From this one quickly obtains the well known
uniqueness result that $\mu=\delta_{J_0}$ if $A=1$. We don't want to give any details
here, but see Proposition \ref{P9.2} below and its discussion for a more
general argument of this type. The material from \cite{Remuniq} is also closely related
to these issues.

We already mentioned several times the fact that $w$ provides a conformal map from
$\C^+$ onto its image. It is advantageous not to work with $w$ itself but
with a related function that is obtained by changing variables, as follows.
Notice that
\begin{equation}
\label{zzeta}
\zeta \mapsto z = z(\zeta) = -\zeta - \frac{1}{\zeta}
\end{equation}
defines a conformal map from the upper semidisk $D^+=D\cap \C^+$ onto
$\C^+$. Here, $D=\{ z: |z|<1\}$ denotes the unit disk. We can therefore introduce
\[
F: D^+\to D^+, \quad\quad F(\zeta) = e^{w(z(\zeta))} .
\]
$F$ indeed maps to the upper unit disk because $\textrm{Re }w<0$, $0<\textrm{Im }w<\pi$.
\begin{Proposition}
\label{P2.2}
$F$ has a holomorphic extension to $D$, by reflection: $F(\overline{\zeta})=
\overline{F(\zeta)}$. This extended function $F$ is
a conformal map from $D$ onto $F(D)\subset D$, with $F(0)=0$, $F'(0)=A$.
\end{Proposition}
Since, at least in general, there is some potential for confusion associated
with this terminology, we should perhaps clarify our use of language here:
by a \textit{conformal map }(also known as a biholomorphic map) we mean a holomorphic bijection
between connected open sets (also called \textit{regions }or \textit{domains});
in fact, all domains in this paper will be simply connected.
\begin{proof}
It's easy to check that if $\zeta_n\in D^+$, $\zeta_n\to x\in (-1,1)$, $x\not=0$, then
$\textrm{Im }F(\zeta_n)\to 0$. Indeed, if $-1<x<0$, say, then $z_n=-\zeta_n-1/\zeta_n \to t>2$,
and thus $k(z_n)\to 1$, by Proposition \ref{P2.1}(a). Since $F=e^{-\gamma}e^{i\pi k}$,
this gives the claim in this case. In fact, part (c) of the Proposition shows us that
$\gamma$ is continuous near such a $t$, so $F$ actually approaches a negative limit.
The case $0<x<1$ is similar; this time, $F$ converges to a positive limit.

The Schwarz reflection principle therefore provides a holomorphic extension of $F$ to
$D\setminus \{ 0\}$. To define $F$ for $\zeta\in D^-$, we refer to the
identity (``reflection'') $F(\overline{\zeta})=\overline{F(\zeta)}$.
Moreover, if $\zeta\in D^+$, $\zeta\to 0$, then $z=-\zeta-1/\zeta$ satisfies $|z|\to\infty$, so
\[
w(z)=\ln A - \int\ln(t-z)\, dk(t) = -\ln(-z) + \ln A + O(1/z)
\]
as $\zeta\to 0$ and this leads to $F(\zeta) = A\zeta + O(\zeta^2)$. It follows that the singularity
at $\zeta=0$ is removable and $F'(0)=A$, as claimed.

Finally, notice that $w$ is a conformal map from $\C^+$ onto its image.
This simply follows from the fact that
$\textrm{Im }w'(z)>0$ on $\C^+$, which we already observed (and used) in the proof
of Proposition \ref{P2.1}.
It now becomes clear that $F$ also maps $D^+$
\textit{injectively }onto a subset of $D^+$ and $D^-$ in the
same way onto the corresponding reflected subset of $D^-$. Moreover, as we observed above,
$F(I)\subset I$ for both $I=(-1,0)$ and $I=(0,1)$. Thus $F$ could fail to be injective
only if $F(x_1)=F(x_2)$ for some points $x_1, x_2$ that are either both in $(-1,0)$ or
both in $(0,1)$. However, it's easy to confirm that $\gamma(-x-1/x)$ is strictly increasing and decreasing,
respectively, on these intervals. Hence $F$ is a conformal map, as claimed.
\end{proof}
We remark in passing that the Schwarz Lemma now provides another simple proof of Lemma \ref{L2.5}.
\begin{Proposition}
\label{P2.3}
(a) The domain $\Omega:=F(D)\subset D$ is of the following type:
If $Re^{i\alpha}\in\Omega$, then
$re^{i\alpha}\in\Omega$ for all $r<R$. Also, $re^{i\alpha}\in\Omega$ if and only if
$re^{-i\alpha}\in\Omega$.

(b) A subset $\Omega\subset D$, $\Omega\not=\emptyset$ is open and
has the properties stated in part (a)
if and only if
there exists an upper semicontinuous
function $h: S^1\to [0,1)$, with $h(e^{-i\alpha})=h(e^{i\alpha})$, so that
\[
\Omega = \Omega_h\equiv \{ re^{i\alpha} : 0\le r <1-h(e^{i\alpha}) \} .
\]
\end{Proposition}
In other words, $\Omega$ is the unit disk with radial slits
\[
S_{\alpha}=\{ re^{i\alpha} : 1-h(e^{i\alpha})\le r\le 1 \}
\]
removed; the function
$h(e^{i\alpha})$ records the height
of the slit at angle $\alpha$.
\begin{proof}
(a) We first discuss the corresponding claim about
the region $w(\C^+)\subset \{ u+iv: u<0, 0<v<\pi \}$.
Fix $v$ and put
\[
L_v=\{ u\in\R : u+iv\in w(\C^+) \} .
\]
We want to show that $L_v = (-\infty, u_0(v))$. If this were not true, then either
$L_v=\emptyset$, or there is an interval $(a,b)\subset L_v$,
with $a,b\notin L_v$ and $a<b\le 0$. This follows because $L_v$ is open, so if
this set is non-empty and not just a half-line, then we can take some other component,
which will necessarily be bounded. Now $L_v=\emptyset$ is clearly impossible
because $k$ takes values arbitrarily close to $0$ and also other values that come
arbitrarily close to $1$, and $w(\C^+)$ is connected.

Take preimages, that is, write $u+iv=w(z(u))$ for
$a<u<b$, and with $z(u)\in\C^+$.
Clearly, $z(u)\equiv x(u)+iy(u)$ is a continuous function of $u\in (a,b)$. Moreover,
$y(u)$ is injective. This follows because $\textrm{Im }w'(z)>0$, as we observed above, so
\begin{equation}
\label{0.71}
\frac{\partial k (x+iy)}{\partial x} > 0 .
\end{equation}
Hence it is not possible for
two points $z_1, z_2$ with the same imaginary part to have images $w(z_1), w(z_2)$
whose imaginary parts agree also.

So $y(u)$ must be monotone, and in fact it's not hard to check that
$y(u)$ is strictly decreasing on $(a,b)$ (but we don't really need to know this here since
an analogous argument would work for strictly increasing $y(u)$).
Notice also that the $z(u)$ stay inside a bounded set, because
$\gamma(z)\to\infty$ as $|z|\to\infty$. Thus, on a suitable sequence $u_n\to a$,
we have that $z(u_n)\to z=x+iy$, and here $y>0$. It follows that $a+iv=w(z)\in w(\C^+)$,
but this contradicts our choice of $a$.

By transforming back to $F$ and $\Omega$, we now obtain the first property of
$\Omega$ for $0<\alpha<\pi$, and, by reflection, also for $-\pi<\alpha<0$.
Here, we have already made use of the invariance of $\Omega$ under
reflection about the real line, but this second property is really obvious from the
corresponding symmetry of $F$.

Next, consider $\alpha=0$. If we recall our discussion of the mapping properties of $F$
from the proof of Proposition \ref{P2.2}, then we see that the positive values of $F(\zeta)$
come from the $\zeta\in (0,1)$. For these $\zeta$, the variable $z=x=-\zeta-1/\zeta$
varies over $(-\infty, -2)$, so
\begin{equation}
\label{0.1}
\Omega\cap (0,1)= \{ e^{-\gamma(x)} : x<-2 \} .
\end{equation}
The Thouless formula (Proposition \ref{P2.1}(c)) shows that $\gamma(x)$ is strictly decreasing
on $x<-2$, and $\gamma(x)\to\infty$ as $x\to -\infty$, so this set is a ray
$(0,R)$, as claimed. The argument for $\alpha=\pi$ is, of course, analogous.

(b) Any domain $\Omega$ with the properties just established is equal
to a domain $\Omega_h$, if we simply define
\begin{equation}
\label{2.31}
h(e^{i\alpha}):= \sup \{ r\ge 0 : (1-r)e^{i\alpha}\notin\Omega \} .
\end{equation}
Furthermore, it is also clear that only this choice of $h$ can possibly work
if it is our goal to represent a given $\Omega$ as an $\Omega_h$ for some $h$.

Conversely, given any function $h:S^1\to [0,1)$, we can form the set $\Omega_h$.
This set will always contain $0$. It is open if and only if
$h$ is upper semicontinuous, and it
is invariant under reflection about the real line
if and only if $h$ is symmetric.
Thus, given $\Omega$ as described in part (a),
$h$ defined by \eqref{2.31} has these properties. Conversely, if an upper semicontinuous,
symmetric $h$ is given, then $\Omega_h$ will be as described in (a).
\end{proof}
We are now in a position to
appreciate why it was useful to change variables and work with
$F$ and $\Omega=F(D)\subset D$ rather than $w$ and $w(\C^+)\subset S=\{ x+iy: x<0,
0<y<\pi \}$. Since always $F(0)=0$, $F'(0)>0$, the conformal map $F$ can be
reconstructed, at least in principle, from its image $\Omega=F(D)$. This is not
true for $w$. Indeed, if $\mu=\delta_{AJ_0}$, where $J_0$ denotes the free Jacobi matrix
$a_n\equiv 1$, $b_n\equiv 0$, then $w_A(z)=w_0(z/A)$, and $w_0$ maps $\C^+$ onto
the full strip $S$. This latter statement
follows easily without any calculation from simple properties of $dk$ and $\gamma$
for the free Jacobi matrix $J_0$, but one can also use the explicit formula
\[
w_0(z)=\ln \left( -\frac{z}{2}+\sqrt{\frac{z^2}{4}-1}\right)
\]
instead. Here, we would have to clarify the precise definitions of the logarithm and the
square root, but in fact a much more transparent formulation is obtained if we just say that
$F_0(\zeta)=\zeta$.

So $w_A(\C^+)=S$ for all $0<A\le 1$, and the
image under $w$ does \textit{not }distinguish between these $w$ functions.
The domains $\Omega_A\subset D$, on the other hand, have slits at $\alpha=0,\pi$
of $A$ dependent heights, so are not equal to one another.
One can verify directly that these slits become invisible if we transform
back to $w$ and $z$. Theorem \ref{T6.1} below will throw some additional light
on this issue.

The slit height function $h$ is closely related to the Lyapunov exponent $\gamma$.
In fact, it is essentially $\gamma$, plus the change of variables $F=e^w$, $\alpha=\pi k(t)$.
\begin{Theorem}
\label{T2.1}
For $0\le \alpha \le \pi$, we have that
\[
h\left( e^{i\alpha}\right) = 1- e^{-L(\alpha)} ,
\]
where
\[
L(\alpha) = \sup \{ \gamma(t): -2\le t\le 2, \pi k(t)=\alpha \} .
\]
\end{Theorem}
If $t\in E=\textrm{\rm top supp }dk$ and $t$ is not an endpoint
of a component $(a,b)\subset (-2,2)\setminus E$, then there is no $s\not=t$ with $k(s)=k(t)$ and thus
for these $t$, the formula above takes the simpler form
\[
h\left( e^{i\pi k(t)} \right) = 1-e^{-\gamma(t)} .
\]
Recall in this context that $\textrm{top supp }dk$,
the \textit{topological support }of $dk$, is defined
as the smallest closed subset $E\subset\R$ with $k(E^c)=0$.

Also, the set $k^{-1}(\{ \alpha/\pi \})\cap [-2,2]$
is either a single point or a closed interval $[a,b]$, because
$k(t)$ is increasing and continuous. In the second case, the interior
$(a,b)$ is a component of $(-2,2)\setminus E$.
\begin{proof}
It is again more convenient to discuss the analogous claim about
the region $w(\C^+)$. So, for $0<v<1$, define
\[
H(v)=\sup \{ u\ge 0 : -u+i\pi v \notin w(\C^+) \} .
\]
We want to show that
\begin{equation}
\label{2.1}
H(v) = L(v) .
\end{equation}

Now for any $t\in (-2,2)$ with $k(t)=v$, we certainly have that
$-\gamma(t)+i\pi v\notin w(\C^+)$. Indeed, if $-\gamma(t)+i\pi v=w(z_0)$ for
some $z_0\in\C^+$, then, by open mapping, the image of a small disk $D_r(z_0)$ under $w$ would include
a disk about $-\gamma(t)+i\pi v$, but at least some of these points also occur as images of
$t+iy$ for small $y>0$, and this contradicts the fact that $w$ is injective. Thus $H(v)\ge L(v)$.

On the other hand, if we had that $H(v)>L(v)$, say
\begin{equation}
\label{0.4}
H(v)\ge\gamma(t)+\epsilon
\end{equation}
for all $t\in (-2,2)$
with $k(t)=v$, then we can again look at the preimages of $-u+i\pi v$ for $u>H(v)$. As in the proof of
Proposition \ref{P2.3}, write $-u+i\pi v=w(z(u))$.
We now let $u$ approach $H(v)$. As above, the $z(u)$ will stay inside a bounded set, so
will converge to a limit $t_0\in\C^+\cup\R$ along a suitable subsequence. In fact,
$t_0\in\C^+$ is impossible here because then $-H(v)+i\pi v=w(t_0)$ would lie
in $w(\C^+)$. Thus $t_0\in\R$. Since $k$ is continuous on $\C^+\cup\R$, we can conclude that
$k(t_0)=v$, and now \eqref{0.4}
demands that $\gamma(t_0)\le H(v)-\epsilon$. The function $\gamma$ is upper semicontinuous,
so this inequality would prevent $u=\gamma(z(u))$ from approaching $H(v)$ when we send
$u\to H(v)$ along the subsequence chosen above. We can escape this absurd situation
only by abandoning \eqref{0.4}.
We have established \eqref{2.1}.

This gives the Theorem for $\alpha\not= 0, \pi$. The remaining cases $\alpha=0,\pi$
do not pose any problems; it suffices to refer to what we discussed already above.
See especially \eqref{0.1}.
\end{proof}
\section{Data sets}
Let us summarize: Starting out from an invariant measure $\mu\in\mathcal M_0$
on $\mathcal J_2$, we
introduced the density of states $dk$ as the average of the spectral measures $\rho_0$
and $\ln A=\int\ln a_0\, d\mu>-\infty$. These have the property that
\begin{equation}
\label{cond1}
-\ln A + \int_{[-2,2]} \ln |t-z|\, dk(t) \ge 0 \quad \textrm{for }z\in\C .
\end{equation}
We then introduced a variety of additional data, which were computed from
$(A,dk)$. We will now show that we
can go back and forth between these.
More precisely, each of the following is determined by
and will determine $(A,dk)$:
\begin{itemize}
\item the $w$ function $w(z)$ on $z\in\C^+$;
\item the Lyapunov exponent $\gamma(z)$ on $z\in\C^+$;
\item the conformal map $F:D\to D$;
\item the image domain $\Omega=F(D)$;
\item the slit height function $h$
\end{itemize}
We will also identify the classes of objects obtained in this way. For easier reference,
we give names to the corresponding sets.
\begin{Definition}
\label{D3.6}
We say that:\\[0.2cm]
(1) $(A,dk)\in\mathcal D$ \textit{(density of states) }if $A>0$
and $dk$ is a probability measure on the Borel sets
of $[-2,2]$ and \eqref{cond1} holds;\\[0.2cm]
(2) $W\in\mathcal W$ \textit{($w$ function) }if $W, W'$ are
Herglotz functions, $W$ maps $\C^+$ to the strip
$S=\{x+iy: x<0, 0<y<\pi\}$, $W'$ extends holomorphically to $\C\setminus [-2,2]$
by reflection $W'(\overline{z})=\overline{W'(z)}$ and $\lim_{y\to\infty} yW'(iy)=i$;\\[0.2cm]
(3) $\Gamma\in\mathcal L$ \textit{(Lyapunov exponent) }if $\Gamma, \partial \Gamma/\partial y$
are positive harmonic
functions on $\C^+$, $\Gamma$ extends harmonically to $\C\setminus [-2,2]$ by reflection
$\Gamma(\overline{z})=\Gamma(z)$, and
\[
\lim_{y\to\infty}\frac{\Gamma(iy)}{\ln y} = 1 ;
\]
(4) $G\in\mathcal C$ \textit{(conformal map)}
if $G:D\to\Omega$ is a conformal map onto a region $\Omega\subset D$
of the type described in Proposition \ref{P2.3}, with $G(0)=0$, $G'(0)>0$;\\[0.2cm]
(5) $\Omega\in\mathcal R$ \textit{(region) }if $\Omega\subset D$
is a region of the type described in
Proposition \ref{P2.3};\\[0.2cm]
(6) $g\in\mathcal H$ \textit{(height function) }if $g:S^1\to [0,1)$
is a symmetric ($g(e^{-i\alpha})=g(e^{i\alpha})$)
upper semicontinuous function.
\end{Definition}
\begin{Theorem}
\label{Tdata}
If $(A,dk)\in\mathcal D$ is given, then the associated data $w$, $\gamma$, $F$, $\Omega$, $h$ have
the properties from parts (2)--(6) of Definition \ref{D3.6}. Conversely, if an
object of one of these types is given, then there exists a unique pair $(A,dk)\in\mathcal D$ that
is associated with it.
\end{Theorem}
At this point, this statement seems to be of conditional type because we have not yet shown
that every $(A,dk)\in\mathcal D$ is actually obtained from an invariant measure $\mu\in\mathcal M_0$,
and indeed, we will leave this issue completely aside in this section and the next.
However, as we will discuss later, this statement is true; see Theorem \ref{T5.1} below.
For now, it will be important to observe that nowhere in the developments that started
with Proposition \ref{P2.2} did
we use the fact that $(A,dk)$ were obtained from a $\mu\in\mathcal M_0$; rather, it was
only property \eqref{cond1} that mattered. Similarly, Proposition \ref{P2.1} continues to
hold if we just assume \eqref{cond1}.

We again witness the effect that things become particularly transparent on the level
of the conformal maps. Note, for instance, that items (2), (3) from
Definition \ref{D3.6} come with a sizeable amount of fine print,
and contrast this with the satisfying fact that all symmetric upper semicontinuous
functions occur as slit height functions.

In one part of the proof, we will make use of several classical results on conformal
maps and their boundary values. This material will also be important in subsequent
sections, so let us give a brief review now.

The first tool is the notion of \textit{kernel convergence }for the image domains $\Omega$.
For a careful discussion of this topic in a general setting, please see
\cite[Section 15.4]{Con}. We give the basic definition
in the form most suitable for our purposes here, and specialized to the case that
is of interest to us.
\begin{Definition}
\label{D3.1}
Let $\Omega_n, \Omega \subset D$ be subdomains of the unit disk of the type discussed
in Proposition \ref{P2.3}. We say that $\Omega_n\to\Omega$ in the sense of
\textit{kernel convergence }if:\\
(i) If $z\in\Omega$, then there exist a radius $r=r(z)>0$ and an index
$N=N(z)$ so that $D_r(z)\subset\Omega_n$ for all $n\ge N$.\\
(ii) If $z\notin\Omega$ and $r>0$ are given, then there exists $N=N(z,r)$ so that
$D_r(z)$ is not contained in $\Omega_n$ if $n\ge N$.
\end{Definition}
To confirm that this is indeed what \cite[Definition 15.4.1]{Con} says
in the present context, observe that the \textit{kernel }with respect to $z_0=0$
(as defined in \cite{Con}) of a
sequence of domains of the type $\Omega_{h_n}$, if it exists,
is another domain of the type $\Omega_h$. In particular, there is no need to take
a specific connected component of the set introduced in \cite{Con}. The general
definition of a kernel also demands that $D_r(0)\subset\Omega_n$ for some $r>0$
and all large $n$, but this is a consequence of (i) here because we always have that $0\in\Omega$.

This notion is important for us here because kernel convergence of the image domains is equivalent
to the locally uniform convergence of the conformal maps from $D$ onto these domains.
We will return to these issues shortly,
but let us first give a characterization of kernel convergence in terms
of the associated slit height functions $h$.
\begin{Lemma}
\label{L3.41}
Let $\Omega_n, \Omega\subset D$ be domains of the type discussed in Proposition \ref{P2.3},
and let $h_n, h$ be the associated slit height functions. Then the following are equivalent:\\
(a) $\Omega_n\to\Omega$ in the sense of kernel convergence;\\
(b) $\sup \varphi h_n \to \sup\varphi h$ for every $\varphi\in C(S^1)$,
$\varphi\ge 0$.
\end{Lemma}
\begin{proof}
We first verify that (b) implies (a).
Let's start with condition (i) from Definition \ref{D3.1}. Fix an arbitrary point
$z\in\Omega_h$, say $z=re^{i\alpha}$. Then $r<1-h(e^{i\alpha})$.
The case $r=0$ is easy: We have that $\sup h <1$, so condition (b) with $\varphi\equiv 1$
shows that also $\sup h_n \le 1-\delta$, uniformly in $n$, for some $\delta>0$, and thus $D_{\delta}(0)
\subset \Omega_{h_n}$ for all $n$. So we can now assume that $0<r<1-h(e^{i\alpha})$.
Since $h$ is upper semicontinuous, we will have that $h\le 1-r-2\epsilon$, say, on a suitable
neighborhood of $e^{i\alpha}$, for some $\epsilon >0$.
We can now use (b) with a function $\varphi$ that is supported
by this neighborhood, equal to $1$ on a smaller neighborhood of $e^{i\alpha}$, and takes
values $0\le\varphi\le 1$. Assumption (b) then says that for all sufficiently large $n$,
we will also have that $h_n(e^{i\beta})\le 1-r-\epsilon$, say, uniformly on some neighborhood
$|\beta-\alpha|\le\eta$. In particular, this shows that $D_{\delta}(z)\subset\Omega_{h_n}$
for all these $n$, if we take $\delta<\min\{\epsilon, r\eta/100 \}$, say.

Let's now move on to condition (ii) from Definition \ref{D3.1}. We are
given a $z\notin\Omega_h$ and a radius $\delta>0$. The assumption that
$z=re^{i\alpha}\notin\Omega_h$ means that $r\ge 1-h(e^{i\alpha})$.
Pick a function $0\le\varphi\le 1$ that is supported by $|\beta-\alpha|\le \delta/10$
and equal to $1$ at $e^{i\alpha}$. Condition (b) then provides angles
$\beta_n$ from this neighborhood so that $h_n(e^{i\beta_n})\ge 1-r-\delta/2$ for all large $n$.
In particular, this shows that $D_{\delta}(z)$ is not contained in $\Omega_{h_n}$ for these $n$,
as desired. This concludes the proof of the implication (b) $\Longrightarrow$ (a).

We now want to show that, conversely, (a) implies (b). Fix $\varphi\in C(S^1)$, $\varphi\ge 0$.
We would first
like to show that
\[
\liminf_{n\to\infty} (\sup \varphi h_n)\ge \sup \varphi h .
\]
The upper semicontinuous
function $\varphi h$ assumes a maximum on the compact set $S^1$,
so $\sup\varphi h = \varphi(e^{i\alpha})h(e^{i\alpha})$
for some $e^{i\alpha}\in S^1$. We may assume here that $\varphi(e^{i\alpha})h(e^{i\alpha})>0$ because
otherwise what we're trying to show is trivially true. In fact, for convenience, let's also assume
that $\varphi(e^{i\alpha})=1$.
We have that $(1-h(e^{i\alpha}))e^{i\alpha}\notin\Omega_h$, and now (ii) from Definition \ref{D3.1}
shows that for
any $\delta>0$ and all large $n\ge N_0=N_0(\delta)$,
we must have $h_n(e^{i\beta})\ge h(e^{i\alpha})-\delta$
somewhere on $\alpha-\delta<\beta<\alpha+\delta$, say.
Since $\varphi$ is continuous, it will satisfy $\varphi\ge 1-\eta$ on this interval,
and here $\eta>0$ can be made arbitrarily small, provided we start the argument with a
sufficiently small $\delta>0$.
Putting things together, we conclude that
$\sup\varphi h_n \ge\sup\varphi h-\delta-\eta$ for all large $n$.
As discussed, $\delta+\eta$ can be made arbitrarily small here, so this is what we wished to show.

It remains to prove that also
\begin{equation}
\label{3.1}
\sup\varphi h \ge \limsup_{n\to\infty} (\sup\varphi h_n) .
\end{equation}
Again, the suprema
are really maxima, attained at $e^{i\alpha_n}$, say. We can now pass to a subsequence on which
we converge to the $\limsup$ from the right-hand side of \eqref{3.1}, and then pass to a
subsequence a second time to make the points converge, say $\alpha_n\to\alpha$.
If \eqref{3.1} were wrong, we would have that $\varphi(e^{i\alpha})h(e^{i\alpha})\le
\varphi(e^{i\alpha_n})h_n(e^{i\alpha_n})-\epsilon_0$, for some
$\epsilon_0>0$ and all large $n$ from the subsequence that was chosen.
Since $\varphi$ is continuous, it would then also follow that
\begin{equation}
\label{3.2}
h(e^{i\alpha})\le h_n(e^{i\alpha_n})- \epsilon ,
\end{equation}
for these $n$ and some new (possibly smaller) discrepancy $\epsilon>0$.
Now obviously $z_0:=(1-h(e^{i\alpha})-\epsilon)e^{i\alpha}\in\Omega_h$,
but \eqref{3.2} says that given any radius $\delta>0$, no matter how small, the corresponding
disk $D_{\delta}(z_0)$ will not be contained in $\Omega_{h_n}$ for infinitely many choices
of $n$. This contradicts condition (i) from Definition \ref{D3.1}.
\end{proof}
The second set of classical results on conformal maps that will play an important role here
deals with the boundary values of these functions. The fundamental
result in its general form says that a conformal map $F: D\to\Omega$ extends to a homeomorphism
$F:\overline{D}\to\widehat{\Omega}$, where $\widehat{\Omega}$ is the union of $\Omega$ with
the collection of its \textit{prime ends, }endowed with a suitable topology. Please see
\cite[Sections 14.2, 14.3]{Con} for a careful discussion; the result just mentioned is stated
as Theorem 3.4 of \cite[Section 14.3]{Con}. For now, we will need the theory of prime
ends only for regions of a relatively simple type; later on, in Section 6,
prime ends will make another
appearance. In both cases, the material from \cite[Sections 14.2, 14.3]{Con}
will provide more than adequate background.

After these digressions, we now return to Theorem \ref{Tdata}. When we prove this,
one assignment will be the task to construct $(A,dk)\in\mathcal D$, given a region
$\Omega\in\mathcal R$. For regions of a certain simple type, this problem admits
an explicit solution, and we will base our treatment of the general case on this.

More precisely, call a domain $\Omega\in\mathcal R$
a \textit{finite gap domain }if the corresponding slit height function $h$ is non-zero
only at finitely many points. So these are regions with finitely many slits;
we call them finite gap domains because they correspond to \textit{finite gap Jacobi matrices,}
that is, reflectionless Jacobi matrices whose spectrum is a \textit{finite gap set}
(a disjoint union of finitely many compact intervals of positive length).
\begin{Lemma}
\label{L3.11}
Suppose that $\Omega\in\mathcal R$ is a finite gap domain. Then there exists
a finite gap set $E\subset [-2,2]$
so that $\Omega$ is the region associated with $A=\textrm{\rm cap }E$, $dk=d\omega_E$.
\end{Lemma}
Here, $\textrm{cap }E$ again denotes the logarithmic capacity of $E$, and $\omega_E$ is the
equilibrium measure of $E$. Please see \cite{Ran,SaTot} for background information on potential theory.
The proof will show that $E$ can be obtained as the inverse image of $\partial D$ under the
(extended) conformal map $z\mapsto F=e^{w(z)}$.

Note also that $\int\ln |t-z|\, d\omega_E(t)\ge\ln\textrm{cap }E$ for all $z\in\C$
for a finite gap set $E$, so \eqref{cond1} holds and thus $(A,d\omega_E)$ is an admissible
set of data from the class $\mathcal D$.
\begin{proof}
Let $F:D\to\Omega$ the unique conformal map onto $\Omega$ with $F(0)=0$, $F'(0)>0$.
It is easy to find the set of prime ends for
a finite gap region $\Omega$. We can conveniently identify this set with a set built
from the boundary $\partial \Omega$ as follows. We use two copies of each slit
(minus its end point)
$\{re^{i\alpha}: 1-h(e^{i\alpha})<r\le 1 \}$. Let's call these $S_+(\alpha)$ and
$S_-(\alpha)$. Then there is a natural bijection between the prime ends of
$\Omega$ and the union of these $S_{\pm}$ with the rest of $\partial \Omega$.
Moreover, using this identification, we can also easily describe the topology
of $\widehat{\Omega}$, the union of $\Omega$ and its prime ends.
The topology is in fact the obvious one, if we think of $\widehat{\Omega}$ as
the union of $\Omega$ and its boundary, but with each slit having two ``sides,''
and points from one side of a slit are not close to those from the other side.
More formally, we can say that if $(re^{i\alpha},+)\in S_+(\alpha)$, say, then a neighborhood base
is given by the sets
\[
U_{\epsilon}=\{ pe^{i\beta} : |p-r|<\epsilon, \alpha<\beta<\alpha+\epsilon \} \cup
\{ (pe^{i\alpha},+): |p-r|<\epsilon \}
\]
for small $\epsilon >0$. Of course, similar descriptions are available at other points,
but we will leave the matter at that.

Recall that we know from \cite[Theorem 14.3.4]{Con} that $F$ extends to a homeomorphism
$F:\overline{D}\to\widehat{\Omega}$. In particular, $F$ maps $\partial D$ homeomorphically
onto the prime ends of $\Omega$. By mapping the prime ends back to the correponding points
in the complex plane,
we also obtain a continuous map $F_0$ from $\partial D$ onto $\partial\Omega$ (the boundary is
now taken as a subset of $\C$). This map is not a homeomorphism; every point on a (half-open)
slit has two preimages.
The inverse image of $\partial D$ under this map $F_0$ is a finite disjoint union of subarcs
of $\partial D$; the number of subarcs is equal to the number of slits.

We now transform back to a putative $W$ function, using the change of variables from Section 2.
Observe that
since $\Omega$ is invariant under reflection about the real axis, so is $F$: we have
that $F(\overline{\zeta})=\overline{F(\zeta)}$. This implies that $F(D\cap\R)\subset D\cap\R$,
and since $F'(0)>0$, we also see that $F(D^+)\subset D^+$, $F(D^-)\subset D^-$, where we again abbreviate
$D^{\pm}=D\cap\C^{\pm}$. Thus we can take a holomorphic logarithm on $D^+$ and define
$W(z)=\ln F(\zeta)$, with $0<\textrm{Im }W(z)<\pi$ for $z\in\C^+$, and
$z$ and $\zeta$ are related by \eqref{zzeta}. This function $W$ maps $\C^+$ conformally onto
the strip $S=\{x+iy: x<0, 0<y<\pi\}$ with finitely many horizontal slits of the type
\[
S(y,d)=\{ x+iy: -d\le x \le 0 \}
\]
removed.
What we just said about the boundary behavior of $F$ and $F_0$ translates into similar statements
about $W$. More precisely, $W$ extends continuously to the boundary $\partial \C^+=\R$ and
maps $\R$ onto the union of $\partial S$ with the slits $S(y_j,d_j)$. Every
point of $S(y_j,d_j)\setminus\{ -d_j+iy_j \}$ has two preimages, all other boundary points have
one preimage.

The points $z\in\R\setminus [-2,2]$ correspond to $\zeta\in (-1,1)$, which are not in the
boundary of the original domain $D$, but of course that is no problem at all because $F$
is holomorphic there and thus definitely extends continuously. Somewhat greater care is
required to handle possible slits at $\alpha=0,\pi$. Here, we observe that we obtain
precisely one side of such a slit as the image of $F$, restricted to $\overline{D^+}$. This
follows from the reflection symmetry of $F$.

Let
\[
E=W^{-1}\left( \{ iy: 0\le y\le \pi \} \right) .
\]
As explained above, $E$ is a finite gap set; it is the inverse image under \eqref{zzeta}
of a finite disjoint union of closed subarcs of $\partial D$. Since, under \eqref{zzeta},
only the $z\in [-2,2]$ produce values $\zeta\in\partial D$, we also know that
$E\subset [-2,2]$.

Next, write $W=-\gamma+i\pi k$. Since $k$ is continuous up to the real line,
the Herglotz representation of this function reads
\[
W(z) = C_0 +Bz + \int_{-\infty}^{\infty} \left( \frac{1}{t-z}-\frac{t}{t^2+1} \right)
k(t)\, dt .
\]
Clearly, the fact that $W$ maps to $S$ forces $B=0$. Moreover, $k$ is an increasing
function. To see this, first recall that $W$ maps the interior of $E$ bijectively
onto $\{iy: 0<y<\pi\}\setminus \{ iy_j\}$. Thus $k$ is monotone on each interval
from $E$. On the other hand, if $(a,b)\subset E^c$, then we have to map this set
under $W$ to the union
of the slits and the top and bottom parts of $\partial S$. As $(a,b)$ is connected,
we in fact have to map to a single such horizontal segment, and we now see that
$k$ is constant on $(a,b)$. Putting things together, we conclude that $k$ is monotone on $\R$.
Finally, arguments $t<-2$ correspond to $\zeta\in (0,1)$, and since $F'(0)>0$, these
get mapped to positive values again under $F$, hence $k(t)=0$ for these $t$. Similarly,
$k(t)=1$ for $t>2$.

To summarize: $k(t)$ is strictly increasing on the interior of $E$ and constant
on each component of the complement, and $k$ increases from $0$ to $1$.
In particular, $k$ generates a probability measure
$dk$ that is supported by $E$.

We can now run the integration by parts calculation from \eqref{wintparts} again.
We obtain that
\begin{equation}
\label{0.2}
W(z) = C - \int_{[-2,2]} \ln (t-z) \, dk(t) .
\end{equation}
This formula was derived for $z\in\C^+$, but it remains valid for the continuous
extension of $W$ to $z\in\R$ because $\textrm{Re }W<0$ on $\C^+$, and now the arguments
from the proof of Proposition \ref{P2.1} yield \eqref{0.2} on $z\in\R$ also.

Let's take a look at
\[
\Phi(z)\equiv -\textrm{Re }W(z)+C =  \int_{[-2,2]} \ln |t-z|\, dk(t) .
\]
From the mapping properties of $W$, we know that $\Phi=C$ on $E$, the support of $dk$,
but $\Phi<C$ on $\C\setminus E$. These properties
identify $\Phi$ as the equilibrium potential of the set $E$ (so $dk=d\omega_E$) and $e^C$ as
the logarithmic capacity of $E$; see \cite[Theorem I.3.1]{SaTot} and also Remark 1.5
from Section I.1 of this reference.

So if we use these data $(A,dk)=(\textrm{cap }E, d\omega_E)\in\mathcal D$ as our input, then we will
obtain the finite gap domain $\Omega\in\mathcal R$ we started out with.
\end{proof}
We are now ready for the
\begin{proof}[Proof of Theorem \ref{Tdata}]
We will focus on the existence part (``Conversely, ...'') exclusively. Indeed,
except for small details, which we leave to the reader to fill in, our discussion
from Section 2
has already shown that the data we introduced have the stated properties. As mentioned above,
it is important to
note here that our arguments only used \eqref{cond1}; it was not essential that in the original setting,
$(A,dk)$ were obtained from a measure $\mu\in\mathcal M_0$.
It is also easy to see that each of the data from
Definition \ref{D3.6} determines $(A,dk)$, so we will not spend any time on uniqueness, either.

With these preliminaries out of the way,
suppose now that a $W\in\mathcal W$ is given. We want to construct
$(A, dk)\in\mathcal D$ so that
\[
W(z)=\ln A - \int\ln(t-z)\, dk(t) .
\]
The properties of $W'$ in particular ensure that $\overline{W'(x)}=W'(x)$ for
$x\in\R\setminus [-2,2]$, that is, $W'$ is real at these points. Therefore, the Herglotz representation
of $W'$ takes the form
\begin{equation}
\label{3.16}
W'(z) = C+Dz + \int_{[-2,2]} \frac{dk(t)}{t-z} ,
\end{equation}
with a finite measure $dk$ and $C\in\R$, $D\ge 0$. In fact, the asymptotics of $W'$
immediately imply that $C=D=0$, $dk(\R)=1$. Thus indeed
\[
W(z) = B - \int\ln (t-z)\, dk(t) .
\]
As usual, we take the logarithm with imaginary part in $(0,\pi)$ here. By assumption
$0<\textrm{Im }W<\pi$ on $\C^+$, and we can now consider $W(Re^{i\alpha})$ with
$0<\alpha<\pi$ and large $R>0$ to conclude that $\textrm{Im }B=0$. In other words,
we can indeed write $B=\ln A$ for some $A>0$, and since also $\textrm{Re }W<0$ by assumption,
it then follows that $(A,dk)$ satisfy condition \eqref{cond1}, as required.

Assume now that we are given a function $\Gamma\in\mathcal L$. The argument, unsurprisingly,
will be quite similar to what we just did.
Introduce $W(z)=-\Gamma(z)+i\pi K(z)$, where
$\pi K$ is a harmonic conjugate of $-\Gamma$ on $\C^+$. This determines
$K$ up to a constant, which will be irrelevant here and can be chosen arbitrarily.
The Cauchy-Riemann equations show that
\[
\textrm{\rm Im }W'(x+iy)=
\pi \frac{\partial K(x+iy)}{\partial x} = \frac{\partial \Gamma (x+iy)}{\partial y} > 0
\]
on $\C^+$. In other words, $W'$ is a Herglotz function.

Consider now the extended function $\Gamma$ on $\C\setminus [-2,2]$. By assumption,
$\Gamma(x+iy)$ is an even function of $y\in\R$ for fixed $|x|>2$. Thus $\partial\Gamma/\partial y$
is odd, and, in particular,
\[
\frac{\partial\Gamma (x+iy)}{\partial y} \Bigr|_{y=0} = 0.
\]
In terms of $W'$, this says that the imaginary part
of this function is zero on $\R\setminus [-2,2]$. Thus the associated measure is supported
by $[-2,2]$ and finite, and we again have a representation of the type \eqref{3.16}.
Integrate and take real parts to obtain that
\[
\Gamma(z) = -Cx - \frac{1}{2}\, D (x^2-y^2) + B + \int_{[-2,2]}\ln |t-z|\, dk(t) .
\]
Since $\Gamma>0$, we must have that $C=D=0$ here, and then the information on the
asymptotics from Definition \ref{D3.6} shows that $dk$ is a probability measure.
We can again write $B=-\ln A$, with $A>0$, and \eqref{cond1} is of course automatic.

In the remaining parts, we will not give a direct construction of $(A,dk)$. Instead,
we will approximate and then make use of compactness properties. More specifically,
recall that we already discussed
the case of a finite gap domain in Lemma \ref{L3.11}, and we will approximate a general
domain by these. So assume now that a $G\in\mathcal C$ is given, let $\Omega=G(D)$
be the corresponding image domain, and denote the associated slit height function by $h$.

Let
\begin{equation}
\label{defh1}
h_n(e^{i\alpha}) = \begin{cases} H_n(j) & \alpha = j\pi/n \quad (j= 0, 1, \ldots, n) \\
0 & \textrm{otherwise} \end{cases} ;
\end{equation}
more precisely, we define $h_n$ by such a formula for $0\le\alpha\le\pi$ and then
extend symmetrically to the lower semicircle. Here, the $H_n(j)$ are defined as follows:
\begin{equation}
\label{defh2}
H_n(j) = \sup_{-1/n\le\delta\le 1/n}
h\left( e^{i\pi (j/n + \delta)}\right) .
\end{equation}
It is then clear that the $h_n$ are slit height functions of finite gap domains $\Omega_n$.
We claim that $h_n\to h$ in the sense that the condition from
part (b) of Lemma \ref{L3.41} holds.
The argument is quite similar to what we did in the second part of the proof of this Lemma.
Let $\varphi\in C(S^1)$, $\varphi\ge 0$. From the definition of $h_n$, we have that
if $h_n(e^{i\alpha})>0$, then
$h_n(e^{i\alpha})=h(e^{i\beta_n})$ for some
$\beta_n=\beta_n(\alpha)$ with $|\beta_n-\alpha|\le \pi/n$. Hence
\[
\varphi(e^{i\alpha})h_n(e^{i\alpha}) = \varphi(e^{i\beta_n})h(e^{i\beta_n}) + R_n(\alpha) ,
\]
and here the error $R_n$ may be estimated by the modulus of continuity of $\varphi$:
\[
\left| R_n\right| \le \omega_{\pi/n}(\varphi)\equiv
\sup_{|\delta|\le \pi/n, \theta\in\R}
\left| \varphi(e^{i(\theta+\delta)}) - \varphi(e^{i\theta}) \right| .
\]
Since $\varphi$ is uniformly continuous on $S^1$, we have that $\omega_{\pi/n}\to 0$ as
$n\to\infty$, and it follows that $\limsup(\sup\varphi h_n)\le\sup\varphi h$.

On the other hand, $\sup\varphi h$ is attained at some point $e^{i\alpha}\in S^1$, and, by construction,
$h_n(e^{i\beta_n})\ge h(e^{i\alpha})$ at some point $|\beta_n-\alpha|\le \pi/n$. Since
$\varphi$ is continuous, this implies that $\liminf(\sup\varphi h_n)\ge\sup\varphi h$.

Lemma \ref{L3.41} now informs us that
$\Omega_n\to\Omega$ in the sense of kernel convergence.
By Carath{\'e}odory's Theorem \cite[Theorem 15.4.10]{Con}, the kernel convergence
of the image domains is equivalent to the locally uniform convergence of the conformal
maps $G_n: D\to\Omega_n$ (normalized, as usual, by agreeing that $G_n(0)=0$, $G'_n(0)>0$),
to the limit $G$.

By Lemma \ref{L3.11}, $G_n(\zeta)=e^{w_n(z)}$ and
\begin{equation}
\label{3.21}
w_n(z)=\ln A_n -\int\ln (t-z)\, dk_n(t)
\end{equation}
for certain data $(A_n,dk_n)\in\mathcal D$ (we actually have much more explicit information
on what these are, but will not use this here).
We can now pass to a subsequence (which, for better readability, we will not make explicit
in the notation) so that $A_n\to A$ and $dk_n\to dk$ in weak-$*$ sense.
Recall in this context that $A_n=G'_n(0)$, and since $G'_n(0)\to G'(0)>0$,
we can be sure that $A>0$. The measure $dk$ is a probability measure on $[-2,2]$.

Taking limits in \eqref{3.21}, we
conclude that
\[
w_n(z)\to w(z) \equiv \ln A - \int\ln (t-z)\, dk(t)
\]
on $z\in\C^+$.
Thus $G_n(\zeta)=e^{w_n(z)}\to e^{w(z)}$, and it follows that $G=e^w$. Put differently,
$G$ is obtained from $(A,dk)$. Since $G(D)\subset D$, it follows that $\textrm{Re }w<0$, so
\eqref{cond1} holds and $(A,dk)\in\mathcal D$.

If a domain $\Omega\in\mathcal R$
or a slit height function $g\in\mathcal H$ is given, we can define an associated
conformal map $G:D\to\Omega$ (with $\Omega=\Omega_g$ in the latter case) and then
use this treatment to again produce a pair $(A,dk)\in\mathcal D$ that corresponds
to the data that were given.
\end{proof}
The question of whether and how compact subsets of $\R$ can be approximated by
periodic spectra (that is, spectra of periodic Jacobi matrices) has received some
attention, and completely satisfactory answers were obtained in at least three independent
works. These are \cite{Bogat,Pehers,Totik} but see also \cite[Sections 5.6, 5.8]{SimRice}
for a comprehensive discussion. In all four cases, the effort needed was not
inconsiderable. The approximation procedure implemented above, see \eqref{defh1}, \eqref{defh2},
together with material that we will discuss in the following section, could be used
to give a tremendously simplified treatment.
\section{Convergence of data}
Most of the data sets introduced in the previous section come with natural
topologies. It seems reasonable to ask what the relations between these are.
\begin{Theorem}
\label{T3.1}
Suppose that $(A_n,dk_n), (A,dk)\in\mathcal D$, and form the associated objects, as above. Then
the following are equivalent:\\
(a) $A_n\to A$ and $dk_n\to dk$ in weak-$*$ sense;\\
(b) $w_n(z)\to w(z)$ locally uniformly on $\C^+$;\\
(c) $\gamma_n(z)\to\gamma(z)$ locally uniformly on $\C^+$;\\
(d) $F_n(\zeta)\to F(\zeta)$ locally uniformly on $D$;\\
(e) $\Omega_n\to\Omega$ in the sense of kernel convergence;\\
(f) $\sup \varphi h_n \to \sup\varphi h$ for every $\varphi\in C(S^1)$,
$\varphi\ge 0$.
\end{Theorem}
\begin{proof}
These statements are either obvious or follow from previously discussed material,
so we can go through this quickly. Clearly, (a) yields pointwise convergence of
the $w$ functions, and a normal families argument then improves this to give the
full claim of (b). Obviously, (b) $\Longrightarrow$ (c).
If (c) is assumed and an arbitrary subsequence is chosen, then we can make
$A_n\to B\ge 0$ and $dk_n\to d\nu$ in weak-$*$ sense on a sub-subsequence (which is not made
explicit in the notation) and then pass to the limit in the Thouless formula along this sequence
to conclude that
\[
\gamma(z) = -\ln B + \int\ln |t-z|\, d\nu(t) .
\]
We now see, first of all, that $B>0$ here, and from the uniqueness of such representations
we in fact infer that $(B,d\nu)=(A,dk)$. So it turns out that $(A,dk)$ is the only possible
limit point of the sequence $(A_n,dk_n)$, and from the compactness property just used we
now obtain (a).

Next, if we recall how $F$ was constructed from $w$, it is also clear that (b) is equivalent
to the locally uniform convergence of $F_n$ to $F$ on $D^+$, which is equivalent to (d),
by a normal families argument.

We already observed earlier that the equivalence of (d) and (e) is exactly what
Carath{\'e}odory's kernel theorem \cite[Theorem 15.4.10]{Con} has to say in the
case at hand. Finally, (e) $\iff$ (f) is Lemma \ref{L3.41}.
\end{proof}
These spaces $\mathcal D$, $\mathcal W$ etc.\ from Definition \ref{D3.6}
become compact if we add a degenerate object, which we can think of as
corresponding to $\mu\notin \mathcal M_0$.
We will discuss this in more detail in a moment. We first
present the analog of Theorem \ref{T3.1} for approach to this added object.
\begin{Theorem}
\label{T3.2}
Let $(A_n,dk_n)\in\mathcal D$, and introduce the corresponding objects,
as in Theorem \ref{Tdata}. Then the following are equivalent:\\
(a) $A_n\to 0$;\\
(b) $|w_n(z)|\to\infty$ locally uniformly on $\C^+$;\\
(c) $\gamma_n(z)\to\infty$ locally uniformly on $\C^+$;\\
(d) $F_n\to 0$ locally uniformly on $D$;\\
(e) $\Omega_n\to \{ 0\}$;\\
(f) $\sup h_n \to 1$.
\end{Theorem}
The condition of part (e) must be interpreted as follows: For every $r>0$, there exists $N$ so that
$D_r(0)$ is not contained in $\Omega_n$ for $n\ge N$.
For example, $\Omega_n = D\setminus [1/n,1)$ converges to $\{ 0\}$ in this sense.

In terms of Carath{\'e}odory's concept of kernel convergence, condition (e) states that
no subsequence $\{ \Omega_{n_j} \}$ has a kernel with respect to $z_0=0$;
see again \cite[Definition 15.4.1]{Con} for background information.
\begin{proof}
This is similar to the previous proof.
It's again easy to see that (a) $\iff$ (b) $\iff$ (c) $\iff$ (d): Indeed,
since $0<k<1$ on $\C^+$, (b) and (c) are obviously equivalent. It is also clear
that (d) implies (c), and conversely, if (c) holds, then at least $F_n\to 0$ locally uniformly
on $D^+$, but that is enough to conclude (d) by a normal families argument again.
Since $A_n=F'_n(0)$, (d) implies (a), and (a) clearly implies (b) and (c).

Obviously, (e) and (f) are equivalent, and thus we can finish the proof by
relating (f) to one of the first four conditions. If (f) is assumed, then Theorem \ref{T2.1}
shows that also $\sup_{-2\le x\le 2}\gamma_n(x)\to\infty$. Since $\gamma_n(x+iy)>\gamma_n(x)$
for $y>0$, this implies that no subsequence of $\gamma_n$ can converge locally uniformly
to a finite harmonic limit function on $\C^+$. A normal families argument now gives (b).

Conversely, if (f) does not hold, say $\sup h_n\le c<1$
on a subsequence, then Theorem \ref{T2.1} shows that there is a corresponding
uniform bound, $\gamma_n(x)\le C$, on $x\in[-2,2]$, along the same
subsequence. So
\[
-\ln A_n + \int\ln |t-x|\, dk_n(t) \le C .
\]
Integrate both sides with respect to $d\omega_0$, the equilibrium measure
of $[-2,2]$ (this will finish the job in a clean way, but is not really necessary;
we could also just integrate with respect to Lebesgue measure on [-2,2]).
Since $\textrm{cap }[-2,2]=1$, we know that $\int\ln |t-x|\, d\omega_0(x)=0$ for
quasi every (in fact: every) $t\in [-2,2]$. Thus Fubini's Theorem yields
the inequality $-\ln A_n\le C$ on the subsequence that was chosen above.
This clearly prevents $A_n$ from converging to zero. We have shown
that (a) does not hold.
\end{proof}

We would like to emphasize one point here that was already made implicitly in our
proof of Theorem \ref{T3.1}.
Consider again a sequence $(A_n,dk_n)\in\mathcal D$, which converges in the sense that
$A_n\to B\ge 0$ and $dk_n\to d\nu$ in weak-$*$ sense. There seem to be three
possibilities: (i) $(B,d\nu)\in\mathcal D$ also, that is, $B>0$ and \eqref{cond1} holds;
(ii) $B>0$, but \eqref{cond1} fails; (iii) $B=0$.

It is very easy to see that (ii) does not occur. This will be used
several times later on, so we state it separately, for easier reference.
\begin{Lemma}
\label{P3.42}
Let $(A_n,dk_n)\in\mathcal D$ and suppose that $A_n\to B\ge 0$ and
$dk_n\to d\nu$ in weak-$*$ sense. Then either $(B,d\nu)\in\mathcal D$ or $B=0$.
\end{Lemma}
\begin{proof}
Suppose that $B>0$. Let $\gamma_n(z)\in\mathcal L$ be the Lyapunov exponents associated
with $(A_n,dk_n)$. Then, by passing to the limit in the Thouless formula,
\[
\gamma_n(z) \to \Gamma(z) \equiv -\ln B + \int\ln |t-z| \, d\nu(t)
\]
for $z\in\C^+$, and since $\gamma_n(z)>0$, we also have that $\Gamma(z)\ge 0$.
This is what \eqref{cond1} is asking for, so $(B,d\nu)\in\mathcal D$, as claimed.
\end{proof}
Finally, let us return to the topic that was already briefly mentioned above:
We can build compact metric spaces starting from the sets
$\mathcal D$, $\mathcal W$ etc.\ from Definition \ref{D3.6}. We first
introduce a metric in such a way that convergence with respect
to this metric is equivalent to the conditions discussed in Theorem \ref{T3.1}.
These spaces are not yet compact, but we can pass to the one-point compactifications
by adding a point at infinity (as the phrase goes); this extended space also admits a
compatible metric, and approach to the point at infinity is then equivalent to the
conditions from Theorem \ref{T3.2}.

There is, of course, general theory underlying this procedure; see, for example,
\cite{Mun}. However, we can also be explicit here and do things entirely by hand. Let us discuss
the space $\mathcal D_0 = \mathcal D \cup \{ 0\}$ in this style (we call the added point $0$ because
it is approached precisely if $A_n\to 0$). We first need a metric on the finite positive
Borel measures on $[-2,2]$ that generates the weak-$*$ topology. Fix such a metric
and call it $D$. Then let
\[
d((A,dk),(A',dk')) =  |A-A'| + D(A\, dk, A'\, dk')
\]
for two points from $\mathcal D$ and $d((A,dk),0)= A+D(A\, dk, 0)$ for the distance to added point $0$;
here, the second argument in $D(A\, dk, 0)$ denotes the zero measure.

This defines a metric $d$ on $\mathcal D_0$ with the desired properties. It follows from
Lemmas \ref{L2.5} and \ref{P3.42} and the compactness of the space of probability Borel
measures $\nu$ on $[-2,2]$ that $\left( \mathcal D_0 , d\right)$ is compact.
Convergence with respect to $d$ is equivalent to the conditions from Theorems \ref{T3.1}(a) and
\ref{T3.2}(a).

We can give similar metrics on the (one-point compactifications of the) other spaces
from Definition \ref{D3.6}. Alternatively, we can just use Theorem \ref{Tdata} and
Theorems \ref{T3.1}, \ref{T3.2} to move things over from $\mathcal D_0$ to those spaces.
We summarize:
\begin{Proposition}
\label{P3.43}
There are metrics on the spaces $\mathcal D_0=\mathcal D \cup \{ 0\}$,
$\mathcal W_0 =\mathcal W \cup \{ \infty \}$ etc.\ so that
convergence with respect to the metric is equivalent to the corresponding statements from
Theorems \ref{T3.1} and \ref{T3.2}, respectively. These spaces are compact.
\end{Proposition}
\section{Existence of invariant measures}
We now come to one of the main points of the whole discussion so far. We also want
to show that given data as in Definition \ref{D3.6}, there exists a shift invariant
measure on $\mathcal J_2$ that produces these data.

For the density of states measure $dk$, this was already shown in \cite{CarKot}.
(Such a result also appears here, as Proposition \ref{Pgamma0}.)
Carmona-Kotani work with an approximation by periodic problems, which is very similar
to what we did above in the approximation procedure that was based
on \eqref{defh1}, \eqref{defh2}. In fact, these approximating data do come from
periodic problems; more generally, finite gap domains yield periodic operators
if all slits are located at angles that are rational multiples of $\pi$.
We cannot guarantee that this method also produces
the correct $A$, and this issue will have to be addressed separately. This difficulty
is directly related to the fact that while the density
of states depends continuously on $\mu$,
the quantity $A$ is, in general, only a semicontinuous function of $\mu$.

Recall that $\mathcal M_0$ was defined as the set of invariant probability measures
on $\mathcal J_2$ with $\ln A_{\mu}\equiv \int\ln a_0\, d\mu>-\infty$. If $\mu\notin\mathcal M_0$,
then we formally set $A_{\mu}=0$.
\begin{Lemma}
\label{LM0}
Suppose that $\mu_n\in\mathcal M_0$ and $\mu_n\to\mu$ in weak-$*$ sense. Then
\begin{equation}
\label{4.11}
A_{\mu} \ge \limsup_{n\to\infty} A_{\mu_n} .
\end{equation}
In particular, $\mu\in\mathcal M_0$ if $\limsup A_{\mu_n}>0$.
\end{Lemma}
The inequality can be strict. For example, if $J_0$ again denotes the
free Jacobi matrix with $a\equiv 1$, $b\equiv 0$ and
\[
\mu_n=\left( 1- \frac{1}{n} \right)\, \delta_{J_0}+ \frac{1}{n}\, \delta_{e^{-n}J_0} ,
\]
then $\mu_n\in\mathcal M_0$, $\mu_n\to\mu=\delta_{J_0}$,
$\ln A_{\mu_n} = -1$ for all $n$, but $\ln A_{\mu}=0$.

As already mentioned above, we may rephrase by saying that the map
$\mu\mapsto A_{\mu}$ is an upper semicontinuous function on the (compact) set of
invariant probability measures on $\mathcal J_2$.

This Lemma is supplemented by Lemma \ref{Ldkcont}, which says that $dk_{\mu_n}\to dk_{\mu}$
in the situation under consideration.
\begin{proof}
Since a limit of invariant measures is invariant itself, the final claim is an immediate
consequence of \eqref{4.11}, so it suffices to prove this inequality. Since
$J\mapsto\ln (a_0(J)+\epsilon)$ is a continuous function on $\mathcal J_2$ for fixed
$\epsilon>0$, we have that
\begin{equation}
\label{4.12}
\int\ln (a_0(J)+\epsilon)\, d\mu_n(J)\to \int\ln (a_0(J)+\epsilon)\, d\mu(J) .
\end{equation}
Moreover, $\int\ln (a_0+\epsilon)\, d\mu\to \ln A_{\mu}\in[-\infty,\infty)$
as $\epsilon\to 0+$ by monotone convergence, so if \eqref{4.11} failed, then we could
find a subsequence and $\epsilon>0$ so that
\[
\int\ln (a_0(J)+\epsilon)\, d\mu(J)\le \int\ln a_0(J)\, d\mu_n(J)-\epsilon
\]
along the subsequence chosen.
However, the integrals on the right-hand side are clearly dominated by
$\int\ln (a_0+\epsilon)\, d\mu_n$, so this contradicts \eqref{4.12}.
\end{proof}
\begin{Proposition}
\label{Pgamma0}
Suppose that $\Gamma\in\mathcal L$. Then there exist $\mu\in\mathcal M_0$ and $d\ge 0$ so that
\[
\Gamma(z) = \gamma_{\mu}(z)+d .
\]
Moreover, if $\inf_{z\in\C^+}\Gamma(z)=0$, then necessarily $d=0$.
\end{Proposition}
Here, $\gamma_{\mu}$ of course refers to the Lyapunov exponent that is constructed
from $\mu\in\mathcal M_0$ as in Section 2, via $(A_{\mu},dk_{\mu})$ and \eqref{defw}.
\begin{proof}
This is similar to the argument we used in the proof of Theorem \ref{Tdata} to construct
$(A,dk)$, given a conformal map $G$. First of all, Theorem \ref{Tdata} provides us
with associated data $(A,dk)$, $W(z)$, $G(\zeta)$, $\Omega$, $h$.
Define again approximating finite gap domains as in \eqref{defh1}, \eqref{defh2},
and denote the corresponding data by $A_n$, $dk_n$, $w_n$ etc. By Lemma \ref{L3.11},
there are finite gap sets $E_n\subset [-2,2]$ so that $A_n=\textrm{cap }E_n$ and
$dk_n=d\omega_{E_n}$.

This approximation procedure is exceedingly useful here
because if $E\subset [-2,2]$ is a finite gap set, then we can give a solution
to the problem we set out to solve, and a fairly explicit one at that. More precisely,
just take any ergodic measure $\mu$ that is supported by
$\mathcal R_0(E)$;
here, $\mathcal R_0(E)$ denotes the set of Jacobi matrices $J$ with
$\sigma(J)=E$ that are reflectionless on $E$; these are usually called \textit{finite gap operators,}
and they have been studied very heavily. An account of the classical theory may be found
in \cite[Chapter 9]{Teschl}, but
see also \cite{PR2,Remac} for much more on the spaces $\mathcal R_0(E)$.
Note that of course $\mathcal R_0(E)\subset \mathcal J_2$ if (and only if) $E\subset [-2,2]$.
Ergodic measures on $\mathcal R_0(E)$
exist because these spaces are compact and shift invariant.

We claim that, as desired, $A_{\mu}=\textrm{cap }E$ and $dk_{\mu}=d\omega_E$ for
such an ergodic $\mu$ on $\mathcal R_0(E)$.
To prove this, it will suffice to show that: (i) $\gamma_{\mu}=0$ almost everywhere
with respect to $\omega_E$; (ii) $dk_{\mu}$
is supported by $E$. Compare the final part of the proof of Lemma \ref{L3.11} for this step,
or, better yet, see Proposition \ref{P9.2} below.

These two properties are well known standard facts about finite gap operators,
so we will be satisfied with just giving a brief review.
First of all, the absolutely continuous part of the spectral measure
$d\rho_0(J)$ is equivalent to $\chi_E(t)\, dt$ for every $J\in\mathcal R_0(E)$, and
this is immediate from the definition of the property of being reflectionless.
See \cite[Chapter 8]{Teschl} or \cite{PR1,Remac} for background. It follows
from (the easy Ishii-Pastur part of) Kotani theory \cite{Kotac}
that $\gamma_{\mu}=0$ (Lebesgue, hence $\omega_E$) almost everywhere on $E$. Alternative arguments are
available, too; for example, \cite{PR1} has a (sketchy, admittedly) discussion of these
issues at the end of the introduction.

Moreover, as $\sigma(J)=E$ for all $J\in\mathcal R_0(E)$,
the spectral measures $\rho_0(J)$ are supported by $E$ and thus
$dk_{\mu}$, being their average, also has this property.

Returning to the main argument, we now have available invariant measures
$\mu_n\in\mathcal M_0$ that produce the (finite gap) data constructed above.
On a suitable subsequence, which we again assume to be the original sequence for notational
convenience, we can make the $\mu_n$ converge to a limiting measure $\mu$ in weak-$*$ sense.
We constructed the approximations so that $h_n\to h$ in the sense of Theorem \ref{T3.1}(f), so
we also have part (a) of the Theorem
and, in particular, $A_n\to A >0$.
Thus Lemma \ref{LM0} guarantees that $\mu\in\mathcal M_0$.
Lemma \ref{Ldkcont} then shows that (on $z\in\C^+$)
\[
\gamma_n(z)=-\ln A_n+\int\ln |t-z|\, dk_n(t) \to -\ln A + \int\ln |t-z|\, dk_{\mu}(t) .
\]
However, from Theorem \ref{T3.1}(c) we know that $\gamma_n$ also converges
to $\Gamma$ locally uniformly on $\C^+$. This gives the representation
$\Gamma=\gamma_{\mu}+d$,
with $d=\ln (A_{\mu}/A)$. If we now recall that $A=\lim A_n$,
then we can use Lemma \ref{LM0} to confirm
that $A_{\mu}\ge A$, so $d\ge 0$, as claimed. The final claim is obvious from this,
since $\gamma_{\mu}\ge 0$.
\end{proof}
This is not completely satisfactory. Of course, we would prefer to be able to
represent $\Gamma=\gamma_{\mu}$, without the shift $d$. To achieve this, we now show
that we can also represent a \textit{larger }function than $\Gamma$, and then take
a suitable convex combination.
\begin{Lemma}
\label{L5.1}
Suppose that $\Gamma\in\mathcal L$. Then there exist $D>0$ and $\mu\in\mathcal M_0$ so that
\[
\Gamma(z) = \gamma_{\mu}(z) - D .
\]
\end{Lemma}
As an immediate consequence of this, we obtain the desired result.
\begin{Theorem}
\label{T5.1}
Suppose that an object as in one of the parts of Definition \ref{D3.6} is given.
Then there exists a $\mu\in\mathcal M_0$ that generates this object.
\end{Theorem}
In other words, if $\Gamma\in\mathcal L$ is given, there exists $\mu\in\mathcal M_0$
so that $\Gamma=\gamma_{\mu}$, or if $(A,dk)\in\mathcal D$ were given, then we can find
$\mu\in\mathcal M_0$ so that
$A=A_{\mu}$ and $dk=dk_{\mu}$ and so forth.

Assuming the Lemma, we can indeed easily establish Theorem \ref{T5.1}, as follows.
First of all, by Theorem
\ref{Tdata}, it suffices to discuss the case where a $\Gamma\in\mathcal L$ is given.
Proposition \ref{Pgamma0} now yields a $\mu_1\in\mathcal M_0$ so that
$\Gamma=\gamma_{\mu_1}+d_1$. If $d_1=0$ here, then we are done. If $d_1>0$, use
Lemma \ref{L5.1} to find $\mu_2\in\mathcal M_0$ and $d_2>0$ so that
$\Gamma=\gamma_{\mu_2}-d_2$. Then
\[
\mu = \frac{d_2\mu_1 + d_1\mu_2}{d_1+d_2}
\]
also lies in $\mathcal M_0$ and satisfies $\gamma_{\mu}=\Gamma$, as desired. So it only remains
to prove Lemma \ref{L5.1}.
\begin{proof}[Proof of Lemma \ref{L5.1}]
By Theorem \ref{Tdata}, we can write
\[
\Gamma(z) = -\ln A + \int_{[-2,2]} \ln |t-z|\, dk(t) ,
\]
for some $(A,dk)\in\mathcal D$. Partition $[-2,2]$ into
$2N$ intervals $I_j$ of length $2/N$
each, ignore those $I_j$ with $c_j:=\int_{I_j}dk(t)=0$, and let
\[
dk_j(t) =\frac{1}{c_j}\, \chi_{I_j}(t)\, dk(t)
\]
for the remaining intervals.
Then we can recover $dk$ as the convex combination $dk=\sum c_j\, dk_j$, and the $dk_j$
are themselves admissible density of states measures because the integrals
$\int \ln |t-z|\, dk_j(t)$ are still bounded below.

So we can define $A_j>0$ by writing
\begin{equation}
\label{5.61}
\ln A_j = \inf_{x\in\R} \int \ln |t-x|\, dk_j(t) ;
\end{equation}
then $(A_j,dk_j)\in\mathcal D$, or, equivalently,
$\Gamma_j\in\mathcal L$, where
\[
\Gamma_j(z) = -\ln A_j + \int \ln |t-z|\, dk_j(t) .
\]
By construction, these new functions all
satisfy $\inf\Gamma_j = 0$. Therefore, Proposition \ref{Pgamma0} provides us
with measures $\mu_j\in\mathcal M_0$ so that $\Gamma_j=\gamma_{\mu_j}$.
Let $\mu = \sum c_j\mu_j$, and also observe that $\ln A_j < -\ln N$; indeed, it suffices
to take $x$ as the center of $I_j$ in \eqref{5.61} to confirm this. We have that
\[
\gamma_{\mu}=\sum c_j\Gamma_j = \Gamma + \ln A - \sum c_j\ln A_j \equiv \Gamma + D ,
\]
and here we can be sure that
\[
D= \ln A - \sum c_j\ln A_j > \ln A + \ln N \cdot \sum c_j = \ln A + \ln N
\]
will be indeed positive, provided we took $N\in\N$ large enough.
\end{proof}
\section{Slits and gaps}
Recall the definitions made in the context of Theorem \ref{T2.1}:
Let $E=\textrm{top supp }dk$ be the topological (= smallest closed) support of $dk$.
$E$ is a compact subset of $[-2,2]$ (with no isolated points),
and thus its complement $(-2,2)\setminus E$
is a disjoint union of open intervals $I_j$, which we call \textit{gaps. }On each gap $t\in I_j$,
the function $k(t)=\int_{[-2,t]}dk(s)$ has a constant value $k_j\in [0,1]$,
which is unique to this gap. We call
$k_j$ the \textit{gap label }of $I_j$.

It is worth pointing out that
$k_0=0$ is a gap label in this sense if and only if $\min E>-2$; the corresponding
gap is the missing piece $(-2,\min E)$. A similar comment applies to $k_0=1$ as a gap label.

We mention in passing that there is an interesting and beautiful theory
(the \textit{Gap Labeling Theorem}) that describes the set of possible gap labels
in terms of the dynamics of the shift map $S$ on $\textrm{top supp }\mu$.
See, for example, \cite{Johns,Schw} for the classical results and
\cite{ABD} for a recent development.

We saw earlier that if $E$ is a finite gap set, then the gap labels correspond exactly
to the slits of $\Omega$. More precisely, $\Omega$ is the unit disk with finitely many
radial slits removed, and these slits are located at the angles $e^{\pm i\pi k_j}$,
with $k_j$ being the gap labels. See Lemma \ref{L3.11} and its proof for these statements.

This correspondence between slits and gaps is valid in general, if we define
the notion of a slit for a general region $\Omega\in\mathcal R$ appropriately.
\begin{Definition}
\label{D6.1}
Let $\Omega\in\mathcal R$, and let $h\in\mathcal H$ be the associated slit height function.
We say that $\Omega$ has a \textit{slit }at angle $e^{i\alpha}$ if
\[
h\left( e^{i\alpha}\right) > \limsup_{t\to 0+} h\left( e^{i(\alpha+\sigma t)}\right)
\]
for at least one of $\sigma=1$ or $\sigma=-1$.
\end{Definition}
So a \textit{slit, }in this technical sense, corresponds to an at least one-sided
jump in the slit height function.
\begin{Theorem}
\label{T6.1}
Let $\Omega\in\mathcal R$ and $0\le\alpha\le\pi$.
Then $\Omega$ has a slit at angle $e^{i\alpha}$ if and only if $k=\alpha/\pi$ is a gap
label of $E=\textrm{\rm top supp }dk$.
\end{Theorem}
\begin{proof}
Suppose first that $k_0\in [0,1]$ is the label of some gap $(a,b)$, with $-2\le a < b\le 2$.
This means that $k(t)=k_0$ for $t\in [a,b]$, but also that $k(t)\not= k_0$ if
$t\in [-2,2]\setminus [a,b]$. In this situation, Theorem \ref{T2.1} says that
$h(e^{i\pi k_0})=1-e^{-\Gamma}$, $\Gamma=\sup_{a\le t\le b}\gamma(t)$.

The Thouless formula shows that $\gamma$ has a harmonic
extension to $\C\setminus E$, and
\[
\gamma''(t)=-\int_E \frac{dk(s)}{(s-t)^2}<0
\]
for $t\in (a,b)$. It also follows,
with the help of monotone convergence, that $\gamma\bigr|_{[a,b]}$ is continuous.
So, in particular, at least one of the inequalities $\Gamma>\gamma(a)$ or
$\Gamma>\gamma(b)$ holds. Let's assume that $\Gamma>\gamma (a)$ and also
that $a>-2$ (if $a=-2$, then $\Gamma>\gamma (b)$ and $b<2$, and an analogous argument
works). Then $\gamma(t)\le\Gamma-\epsilon$
for all $a-\epsilon\le t\le a$ for some small $\epsilon>0$ because
$\gamma$ is upper semicontinuous.
Now $k(a-\epsilon)<k(a)=k_0$, so Theorem \ref{T2.1} implies that
$h(e^{i\pi k})\le h(e^{i\pi k_0})-\delta$
for some $\delta>0$ and all $k<k_0$ that are sufficiently close to $k_0$. This is what
we wanted to show.

To prove the converse, we again use Carath{\'e}odory's theory of the boundary values
of conformal maps. Assume that
\begin{equation}
\label{6.1}
\limsup_{k\to k_0-}h\left( e^{i\pi k}\right) < h\left( e^{i\pi k_0}\right)\equiv h_0 ,
\end{equation}
the other case being analogous, of course. In more geometric terms,
assumption \eqref{6.1} means that $\partial\Omega_h$ contains an exposed line
segment
\begin{equation}
\label{6.3}
S = \left\{ re^{i\pi k_0} : 1-h_0 +\epsilon < r < 1-h_0+2\epsilon \right\}
\end{equation}
that can be accessed
from $\Omega_h$ through smaller angles. Or, more formally, we can choose $\epsilon>0$
so small that also $Q\subset\Omega_h$, where
\[
Q=\left\{ re^{i\alpha} : 1-h_0+\epsilon < r < 1-h_0+2\epsilon, \quad
\pi k_0-\epsilon < \alpha < \pi k_0 \right\} .
\]
As a consequence, each point on $S$ from \eqref{6.3} corresponds to
a different prime end. Let us try to say this in more precise language:
If $z_n$ is a sequence of points from $Q$ that converges
(in traditional sense) to some $z\in S$, then $z_n$, viewed as a sequence
from $\widehat{\Omega}_h$, the union of $\Omega_h$ with its prime ends, with the topology
discussed in \cite[Section 14.3]{Con}, converges to some prime end.
(This is easy to show, but for our purposes here, convergence on
a subsequence is enough, and this is automatic because $\widehat{\Omega}_h$
is compact.) Moreover, and this is actually the crucial part, if $z\not=z'$,
then the corresponding prime ends are different also. This follows immediately
from the way prime ends were defined.
Finally, recall again \cite[Theorem 14.3.4]{Con}, which says that $F$ extends
to a homeomorphism $F:\overline{D}\to\widehat{\Omega}_h$.

The upshot of all this is the following: We can find two sequences $\zeta_n, \zeta'_n\in D$
which converge to two different boundary points $\zeta, \zeta'\in\partial D$, so that
$F(\zeta_n)$, $F(\zeta'_n)$ both converge to points on
$S$ from \eqref{6.3}. We obtain these sequences by simply picking sequences $z_n,z'_n\in Q$ so that
$z_n\to z$, $z'_n\to z'$, and here $z,z'$ are two distinct points from $S$. We then
let $\zeta_n=F^{-1}(z_n)$, $\zeta'_n=F^{-1}(z'_n)$.

In fact, we can and must say slightly more here: Since the $z_n,z'_n$ can all be chosen
from the same semidisk (either $D^+$ or $D^-$), it is also true that $\zeta,\zeta'$ will
either both be on the (closed) upper semicircle, or they will both be on the lower semicircle.

If we now go back to the original variables and recall that
$k(z)$ is continuous on $\C^+\cup\R$ (see Proposition \ref{P2.1}(a)), then this says
that there are $t,t'\in [-2,2]$, $t\not= t'$, with $k(t)=k(t')=k_0$. Thus $k_0$ is a gap label.
\end{proof}
Tools from the classical theory of conformal maps can be used to analyze other questions, too.
For example, \cite[Theorem 14.5.5]{Con} says that $F:D\to\Omega$ has a continuous extension
$F_0:\overline{D}\to\overline{\Omega}$ if and only if $\partial\Omega$ is locally connected.
Note that we are now seeking an extension that takes values in $\C$, so this issue is not
directly addressed by the theory of prime ends. This result may be used to establish the
following criterion for the continuity of the Lyapunov exponent.
\begin{Theorem}
\label{T6.2}
Let $\gamma\in\mathcal L$, and let $h\in\mathcal H$ be the associated slit height function.
Then $\gamma(z)$ is continuous on $\C$ if and only if the following holds:
(i) If $\alpha/\pi$ is not a gap label,
then $h$ is continuous at $e^{i\alpha}$; (ii) if $\alpha/\pi$ is a gap label, then
$\lim_{t\to 0+} h(e^{i(\alpha+\sigma t)})$ exists for both $\sigma=1$ and $\sigma=-1$.
\end{Theorem}
This can be proved by verifying that $\partial\Omega_h$ is locally connected if and only
if (i), (ii) hold. Note that as $k(z)$ is always continuous on $\C^+\cup\R$, the conformal
map $w$ has a continuous extension to this set if and only if $\gamma$ has this property
(and in this case, $\gamma$ extends continuously to all of $\C$, by the Thouless formula).
Also, this condition is of course equivalent to the possibility of extending $F$ continuously
to $\overline{D}$. Having made these remarks, we omit the detailed proof of Theorem \ref{T6.2}.
An alternative, more direct proof that is based on Theorem \ref{T2.1} is also possible.
\section{More on Lyapunov exponents}
In this section, we discuss $\gamma(x)$ as a function on $x\in [-2,2]$.
Potential theory implies that if $\gamma_1(x)=\gamma_2(x)$ for quasi every (that is,
off a set of capacity zero) such $x$,
then $\gamma_1\equiv\gamma_2$. See \cite[Section I.3]{SaTot}.
So this restriction of $\gamma$ to $[-2,2]$ still contains all the
information. We do not have a description of the set of all these functions,
but we are able to offer the following statements, which supplement Theorems \ref{T3.1}, \ref{T3.2}.
\begin{Theorem}
\label{T7.1}
Let $\gamma_n,\gamma\in\mathcal L$. Then the following conditions are also equivalent to
those from Theorem \ref{T3.1}:\\
(a)
\begin{equation}
\label{7.2}
\sup_{-2\le x\le 2} \varphi(x)\gamma(x) =
\lim_{n\to\infty} \sup_{-2\le x\le 2} \varphi(x)\gamma_n(x)
\end{equation}
for all $\varphi\in C[-2,2]$, $\varphi\ge 0$.\\
(b) The $\gamma_n(x)$ ($n\ge 1$, $-2\le x\le 2$) are uniformly bounded,
and if $\nu\in\mathcal P$ (defined below), then
\begin{equation}
\label{7.12}
\lim_{n\to\infty} \int_{[-2,2]} \left| \gamma(x)-\gamma_n(x) \right| \, d\nu(x) = 0 .
\end{equation}
\end{Theorem}
Here, we let $\mathcal P$ be the set of probability measures $\nu$ on the Borel sets
of $[-2,2]$ for which the potential
\begin{equation}
\label{0.6}
\Phi_{\nu}(x)\equiv\int_{\R}\ln |t-x|\, d\nu(t)
\end{equation}
is a continuous
function of $x\in\R$. This in particular forces $\nu$ to give zero weight to all sets
of capacity zero. On the other hand, for any compact $K\subset [-2,2]$ of positive capacity,
there exists a $\nu\in\mathcal P$ with $\nu(K^c)=0$. See \cite[Corollary I.6.11]{SaTot}.
So, in some vague sense, one can perhaps say that
the class $\mathcal P$ is equivalent to capacity.

There are limits to this, however.
More specifically, while the $L^1(\nu)$ convergence from (b) of course implies
convergence in measure with respect to every $\nu\in\mathcal P$, that is,
\begin{equation}
\label{7.21}
\nu(|\gamma-\gamma_n|\ge\epsilon)\to 0 \quad\textrm{for every }\epsilon>0 ,
\end{equation}
we are \textit{not }claiming that the \textit{capacity }of the set
where $|\gamma_n-\gamma|\ge\epsilon$ approaches zero, and indeed this latter statement
is false. A counterexample may be constructed by approximating a positive $\gamma\in\mathcal L$,
say $\gamma(x)\equiv 1$ on $[-2,2]$,
by a sequence of $\gamma_n$'s corresponding to finite gap sets $E_n$, as in
the proof of Theorem \ref{Tdata} (compare \eqref{defh1}, \eqref{defh2}). Lemma \ref{L3.11} then shows that
\[
\textrm{cap}\left( \{ x\in [-2,2] : \gamma_n(x)=0 \} \right) = \textrm{cap }E_n= A_n .
\]
By construction, the $A_n$ approach the positive limit $A=F'(0)$, where
$F\in\mathcal C$ is the conformal map associated with $\gamma$ (so if $\gamma\equiv 1$,
then $F(\zeta)=e^{-1}\zeta$, but we don't need to know this here).
\begin{Theorem}
\label{T7.2}
Let $\gamma_n\in\mathcal L$. Then the conditions from Theorem \ref{T3.2} are equivalent to:\\
(a)
\begin{equation}
\label{7.4}
\lim_{n\to\infty} \sup_{-2\le x\le 2}\gamma_n(x) = \infty .
\end{equation}
(b) If $\nu\in\mathcal P$, then
\[
\lim_{n\to\infty} \int_{[-2,2]}\gamma_n(x)\, d\nu(x) = \infty .
\]
\end{Theorem}
Since \eqref{7.2} and \eqref{7.4} are analogous to conditions (f)
from Theorems \ref{T3.1} and \ref{T3.2}, respectively,
and, moreover, $\gamma$ and $h$ are directly related through changes of variables
(and a partial maximization), as spelled out in Theorem \ref{T2.1},
it seems tempting to try to relate these directly.
We are going to give a different, more indirect argument, however, which seems easier
and more convenient.
\begin{proof}[Proof of Theorem \ref{T7.2}]
We start with this because we will use Theorem \ref{T7.2} in our proof of Theorem \ref{T7.1}.
The equivalence of (a) with the conditions of Theorem \ref{T3.2} is an immediate
consequence of Theorem \ref{T2.1}, which in particular implies
that for any $\gamma\in\mathcal L$, the associated slit height function satisfies
\[
\sup_{0\le\alpha\le\pi} h(e^{i\alpha}) = 1-\exp\left( -\sup_{-2\le x\le 2}\gamma(x)\right) .
\]
So \eqref{7.4} holds if and only if $\sup h_n\to 1$, which is condition (f) from
Theorem \ref{T3.2}.

Next, assume that $A_n\to 0$ (this is (a) of Theorem \ref{T3.2}). We want to
derive (b) from this. Integrate the Thouless formula with respect to $d\nu$.
With the help Fubini's Theorem, this gives
\[
\int_{[-2,2]}\gamma_n(x)\, d\nu(x) = -\ln A_n + \int_{[-2,2]}\Phi_{\nu}(t)\, dk_n(t) .
\]
Here, $\Phi_{\nu}$ is continuous by assumption, hence bounded, and thus the
integrals on the right-hand side stay bounded, and (b) follows.

Finally, if (b) is assumed, then (a) follows trivially.
\end{proof}

In the next proof, we will make repeated use of two fundamental potential theoretic results,
the \textit{lower envelope theorem }and the
\textit{principle of descent. }We will state them here,
but please refer to \cite[Theorems I.6.8, I.6.9]{SaTot} for a fuller discussion.

Suppose that $dk_n\to d\nu$ in weak-$*$ sense. Then
\[
\Phi_{\nu}(x) = \limsup_{n\to\infty} \Phi_n(x)
\]
for quasi every $x\in [-2,2]$ (the \textit{lower envelope theorem}).
Here, the logarithmic potential $\Phi_{\nu}$ of a measure $\nu$ is again defined
by \eqref{0.6}, and we of course further abbreviated $\Phi_n\equiv\Phi_{dk_n}$.

This is supplemented by
the \textit{principle of descent, }which says that
\[
\Phi_{\nu}(z) \ge \limsup_{n\to\infty} \Phi_n(z)
\]
for all $z\in\C$.
Again, this is interesting for $z=x\in [-2,2]$. On the complement of this set, the
stronger property of locally
uniform convergence is obvious.
\begin{proof}[Proof of Theorem \ref{T7.1}]
We first show that the conditions of Theorem \ref{T3.1} imply (a).
Let $\varphi\in C[-2,2]$, $\varphi\ge 0$ be given.
As in the proof of Lemma \ref{L3.41}, we will split \eqref{7.2} into two inequalities.
We first show that
\begin{equation}
\label{7.1}
\sup\varphi\gamma \ge \limsup_{n\to\infty} \left( \sup \varphi\gamma_n \right) .
\end{equation}
Since the functions $\varphi\gamma_n$ are upper semicontinuous,
the suprema are maxima, so if \eqref{7.1} were wrong, we would find ourselves in the
following situation:
\begin{equation}
\label{3.3}
\sup\varphi\gamma \le \varphi(x_n)\gamma_n(x_n) - \epsilon ,
\end{equation}
for all $n$ from a suitable subsequence and certain points $x_n\in [-2,2]$,
and here can also assume that $x_n\to x\in [-2,2]$ along that same sequence.
Let $d\nu_n$ be a shifted version of $dk_n$;
more precisely,
\[
\int f(t)\, d\nu_n(t) = \int f(t+x-x_n)\, dk_n(t)
\]
for $f\in C(\R)$. Notice that $\Phi_{\nu_n}(x)=\Phi_{dk_n}(x_n)$.
By (a) of Theorem \ref{T3.1}, $dk_n\to dk$ in weak-$*$ sense and
thus also $d\nu_n\to dk$ along the subsequence that was chosen above.
Since, furthermore, $A_n\to A$, the principle of descent now says that
\[
\gamma(x)\ge \limsup \gamma_n(x_n)
\]
(the $\limsup$ is taken along some subsequence, but this
is irrelevant here).
Since $\varphi$ is continuous, this contradicts \eqref{3.3}, unless $\varphi(x)=0$.
However, if $\varphi(x)=0$, then \eqref{3.3} implies that $\gamma_n(x_n)\to\infty$,
and we again obtain a contradiction, this time to Theorem \ref{T7.2}.
We have established \eqref{7.1}.

Next, we show that also
\begin{equation}
\label{7.3}
\sup\varphi\gamma \le \liminf_{n\to\infty} \left( \sup \varphi\gamma_n \right) ,
\end{equation}
and this together with \eqref{7.1} will of course establish \eqref{7.2}.
Again, we argue by contradiction.
If \eqref{7.3} failed, then we would find a subsequence and $x\in [-2,2]$ so that
\begin{equation}
\label{3.4}
\varphi(x)\gamma(x) \ge \varphi(t)\gamma_n(t) + 2\epsilon
\end{equation}
for all $t\in [-2,2]$ and all $n$ from that sequence.
We can now use the fact that $\gamma$ is continuous
with respect to the fine topology and slightly change $x$ to obtain another inequality
of this type (with $2\epsilon$ replaced by $\epsilon$, say), where we can now also guarantee
that $x$ is not from the exceptional capacity zero set from the lower envelope theorem.
Thus $\gamma(x)=\limsup\gamma_n(x)$. Here, the
$\limsup$ is taken along the same subsequence that was singled out above (this is important);
in other words, we applied the lower envelope theorem to this subsequence and not
to the original sequence. We obtain a contradiction to \eqref{3.4} with $t=x$.

To prove that, conversely, (a) above implies part (a) from Theorem \ref{T3.1},
we again exploit the compactness properties that were discussed in Sections 4.
Suppose that \eqref{7.2} holds.
We can pass to a subsequence so that $A_n\to B$, $dk_n\to d\nu$. Here, by Lemma \ref{P3.42},
either $B=0$ or $(B,d\nu)\in\mathcal D$. The first case is impossible because then
Theorem \ref{T7.2} would imply
that \eqref{7.4} holds on the subsequence we chose, but this
is clearly incompatible with our assumption that we have \eqref{7.2}.

So $(B,d\nu)\in\mathcal D$, but then, by what we showed already,
\begin{equation}
\label{7.6}
\lim \left( \sup\varphi\gamma_n \right) = \sup\varphi\gamma_{(B,d\nu)}
\end{equation}
along the subsequence constructed, for all $\varphi\in C[-2,2]$, $\varphi\ge 0$.
However, limits in this
sense are unique. In other words, if $\gamma,\widetilde{\gamma}\in\mathcal L$ are not the
same function, then
\begin{equation}
\label{7.5}
\sup_{-2\le x\le 2} \varphi(x)\gamma(x) \not=
\sup_{-2\le x\le 2} \varphi(x)\widetilde{\gamma}(x)
\end{equation}
for some nonnegative $\varphi\in C[-2,2]$.
Indeed, if $\gamma(x_0)<\widetilde{\gamma}(x_0)$, say, for some $x_0\in [-2,2]$, then, as
$\gamma$ is upper semicontinuous, we in fact have that
$\gamma(x)\le\gamma(x_0)-\epsilon$ for all $x$
from some neighborhood of $x_0$ also, so we can simply
take a $\varphi$ that is supported by this neighborhood,
$0\le\varphi\le 1$, and $\varphi(x_0)=1$, and we are then guaranteed that \eqref{7.5} holds.

This uniqueness means that \eqref{7.6} forces $\gamma_{(B,d\nu)}$ to be the
function $\gamma$ from \eqref{7.2}, and thus, by the uniqueness part of Theorem \ref{Tdata},
$(B,d\nu)=(A,dk)$, the data associated with $\gamma$. So this is the only possible limit
point of the sequence $(A_n,dk_n)$, but any subsequence has a limit point, thus the whole
sequence has to approach this limit, and this is condition (a) from Theorem \ref{T3.1}.

Next, we again assume the conditions from Theorem \ref{T3.1}, and we now wish to
establish (b). First of all, we certainly have that $\gamma_n(x)\le C$ for all $n,x$ and some
uniform bound $C$. We have already shown that \eqref{7.2} holds under the present assumptions,
so we can now obtain this uniform bound very conveniently by just taking $\varphi\equiv 1$
in this condition.

So we can focus on \eqref{7.12}. Fix a $\nu\in\mathcal P$.
We will show that $\gamma_n\to\gamma$ in measure, that is,
\eqref{7.21} holds. This is sufficient because, as just discussed, $0\le\gamma_n,\gamma\le C$,
so $L^1(\nu)$ convergence will follow from this.

We will argue by contradiction and thus assume hypothetically that \eqref{7.21} fails.
Then there exists $\epsilon>0$ so that
\begin{equation}
\label{7.51}
\nu( |\gamma-\gamma_n|\ge\epsilon)\ge\epsilon
\end{equation}
for all $n$ taken from some subsequence.

Recall that $\gamma(x)\ge\limsup\gamma_n(x)$ for all $x$ by the principle of descent. So if
we are given an $\eta >0$, we can find an integer $N=N(x,\eta)$ so that
\[
\gamma_n(x)\le\gamma(x) + \eta \quad\quad \textrm{for all }n\ge N .
\]
We can also choose these integers $N(x,\eta)$ as a measurable function of $x\in [-2,2]$. Then
$\nu(N>N_0)\to 0$ as $N_0\to\infty$ by monotone convergence, so we can in fact find a
(constant) integer $N_0$ and an exceptional set $\mathcal E\subset [-2,2]$ with
$\nu(\mathcal E)<\eta$ so that
\[
\gamma_n(x)\le \gamma(x) + \eta
\]
whenever $n\ge N_0$ and $x\notin\mathcal E$. If we take $\eta<\epsilon/2$, say, then
\eqref{7.51} now has the more specific consequence that
\[
\nu( \gamma-\gamma_n\ge\epsilon)\ge \frac{\epsilon}{2}
\]
for all $n\ge N_0$ from the sequence that was determined earlier. Abbreviate
\[
S_n = \{ x\in [-2,2] : \gamma_n(x)\le\gamma(x)-\epsilon \} ;
\]
then, as just observed, $\nu(S_n)\ge \epsilon/2$ for these $n$.
It follows that
\begin{align*}
\int_{[-2,2]}\gamma_n(x)\, d\nu(x) & = \int_{S_n}\gamma_n(x)\, d\nu(x) +
\int_{S_n^c}\gamma_n(x)\, d\nu(x) \\
& \le \int_{S_n}\gamma(x)\, d\nu(x) - \frac{\epsilon^2}{2} +
\int_{S_n^c}\gamma_n(x)\, d\nu(x)\\
& \le \int_{[-2,2]}\gamma(x)\, d\nu(x) + (C+1)\eta - \frac{\epsilon^2}{2} .
\end{align*}
To obtain the last line, we further split
$S_n^c$ into two parts. On $S_n^c\cap\mathcal E^c$, we have the inequality
$\gamma_n\le\gamma+\eta$, so this part of the integral may be estimated
by $\int_{S_n^c}\gamma\, d\nu + \eta$, and on $S_n^c\cap\mathcal E$,
we just use that $\gamma_n\le C$ and $\nu(\mathcal E)<\eta$.

If we took $\eta>0$ so small that $(C+1)\eta < \epsilon^2/2$, then this says that
$\int\gamma_n\, d\nu \le \int\gamma\, d\nu -\delta$ for some $\delta>0$ and all
$n$ from a certain subsequence. This is impossible because we can also show that
$\int\gamma_n\, d\nu\to \int\gamma\, d\nu$. This is done as above, by integrating
the Thouless formula and using Fubini's Theorem:
\begin{align*}
\int_{[-2,2]}\gamma_n(x)\, d\nu(x) & = -\ln A_n + \int_{[-2,2]}\Phi_{\nu}(t)\, dk_n(t) \\
& \to -\ln A + \int_{[-2,2]}\Phi_{\nu}(t)\, dk(t)\\
& = \int_{[-2,2]}\gamma(x)\, d\nu(x) ,
\end{align*}
because $A_n\to A>0$ and $dk_n\to dk$ in weak-$*$ sense by assumption, and, also by
assumption, $\Phi_{\nu}$ is a continuous function.
This contradiction proves \eqref{7.21}.

Conversely, if (b) is assumed, we repeat the argument from above: Consider any subsequence on
which $A_n\to B\ge 0$, $dk_n\to d\rho$. We want to show that then necessarily
$B=A>0$, $d\rho=dk$, where $(A,dk)\in\mathcal D$ are the data of $\gamma$.
As above, $B=0$ is impossible because then Theorem \ref{T7.2}(b) would apply on the
corresponding subsequence, and this is incompatible with our assumption that
\eqref{7.12} holds. So $(B,d\rho)\in\mathcal D$ by Lemma \ref{P3.42}.
As a consequence, by what we showed already, $\gamma_n\to\gamma_{(B,d\rho)}$ in $L^1(\nu)$
along the corresponding subsequence.
Thus $\gamma_{(B,d\rho)}(x)=\gamma(x)$ almost everywhere with respect to $\nu$ for
all $\nu\in\mathcal P$.
This implies that $\gamma_{(B,d\rho)}(x)=\gamma(x)$
for quasi every $x\in [-2,2]$ because, as we reviewed above, any positive capacity set admits
a measure $\nu\in\mathcal P$ that is supported by it.
We conclude that $\gamma_{(B,d\rho)}=\gamma$ are the same function, thus
$(B,d\rho)=(A,dk)$ by the uniqueness part of Theorem \ref{Tdata}.
\end{proof}
\section{Positive Lyapunov exponents}
In this section, we present a variation on a theme composed by Avila and Damanik \cite{AD}.
These authors show that if an ergodic system is fixed and factors (= homomorphic images)
are considered, then generically the Lyapunov exponent is positive Lebesgue almost
everywhere, with respect to a natural topology.

The material discussed in this paper provides a very natural approach to these issues.
The key fact is the following consequence of Theorem \ref{T7.1}(b).
\begin{Lemma}
\label{L8.1}
Let $\nu\in\mathcal P$.
For any $a,b\ge 0$, the set
\[
S(a,b) = \left\{ \gamma\in\mathcal L : \nu(\gamma\le a)\ge b \right\}
\]
is a closed subset of (the metric space) $\mathcal L$.
\end{Lemma}
Here, we again use the customary self-explanatory notation where a condition is
used to denote the set it defines.
\begin{proof}
Let $\nu\in\mathcal P$.
Suppose that $\gamma_n\in S(a,b)$, $\gamma\in\mathcal L$, $\gamma_n\to\gamma$ in the
sense of Theorem \ref{T3.1}(c) or one of the equivalent descriptions of this mode of
convergence.

Given $\epsilon>0$, no matter how small, Theorem \ref{T7.1}(b), or rather its
consequence \eqref{7.21}, lets us find an integer $N$
and an exceptional set $\mathcal E\subset [-2,2]$, such that $\nu(\mathcal E)<\epsilon$ and
$|\gamma_N(x)-\gamma(x)|<\epsilon$ if $-2\le x\le 2$, $x\notin\mathcal E$. Since $\gamma_N\in S(a,b)$
by assumption, this implies that
\[
\nu(\gamma\le a+\epsilon ) \ge b-\epsilon .
\]
With the help of the monotone convergence theorem, one can now check that this condition
for arbitrary $\epsilon>0$ implies that $\gamma\in S(a,b)$, as desired.
\end{proof}
The Lemma can be rephrased, as follows: The function $\gamma\mapsto\nu(\gamma\le a)$
is upper semicontinuous. Compare this formulation with \cite[Lemma 1]{AD}.
\begin{Corollary}
\label{C8.1}
Let $\nu\in\mathcal P$. Then the set
\[
\left\{ \gamma\in\mathcal L : \gamma(x)>0
\quad\textrm{\rm for }\nu\textrm{\rm -almost every }x \right\}
\]
is a dense $G_{\delta}$ subset of the compact metric space $\mathcal L_0$.
\end{Corollary}
Recall that $\mathcal L_0$ was defined as the one-point compactification of $\mathcal L$;
please review Proposition \ref{P3.43} and its discussion in this context.

The Corollary has further implications because,
by the classical Kotani theory \cite{Kotac}, absolutely continuous
spectrum for ergodic systems corresponds to zero Lyapunov exponents. See \cite{AD} for these aspects
of the Corollary.
\begin{proof}
By Lemma \ref{L8.1}, the sets
\[
U(a,b) =S(a,b)^c = \left\{ \gamma\in\mathcal L : \nu(\gamma>a)>1-b \right\}
\]
are open in $\mathcal L$ and thus also in $\mathcal L_0$. Monotone convergence shows
that $\nu(\gamma>0)=\lim_{a\to 0+}\nu(\gamma>a)$, so the set from the Corollary
may be represented as follows
\[
\bigcap_{n\ge 1} \bigcup_{a>0} U(a,1/n) ;
\]
it is a countable intersection of open sets, as claimed. It is also dense
because for any $\gamma(z) \in\mathcal L$, we have that $\gamma(z)+1/n\in\mathcal L$ also,
and this sequence converges to $\gamma(z)$ in $\mathcal L$. (Approximation of
$\gamma=\infty$ by members of the set from the Corollary is of course a trivial assignment.)
\end{proof}
\section{Ergodic measures}
Return to the discussion of Section 5. We are given a $\Gamma\in\mathcal L$ (or other
data with the properties from Definition \ref{D3.6}), and we constructed an
invariant measure $\mu\in\mathcal M_0$ so that $\Gamma=\gamma_{\mu}$.
We cannot guarantee that $\mu$ will be ergodic here (even if Proposition \ref{Pgamma0}
already provides the correct $\mu$ and we choose the approximating measures $\mu_n$
as ergodic measures, really nothing has been achieved
because a limit of ergodic measures need not be ergodic itself). It is natural
to ask if it is also possible to find an ergodic $\mu$ so that $\Gamma=\gamma_{\mu}$.

Unfortunately, we don't have anything substantially new to say on this interesting
question. Basically, we will review and put into context some observations made by
Kotani in \cite{Kot85}, and then point out some obvious open questions.
\begin{Proposition}
\label{P9.1}
Suppose that $\Gamma\in\mathcal L$ is an extreme point of the convex set $\mathcal L$.
Then there exists an \textrm{ergodic }measure $\mu\in\mathcal M_0$ so that $\Gamma=\gamma_{\mu}$.
\end{Proposition}
This does not come as a big suprise. Ergodic measures are precisely the extreme points
of the set of invariant measures, so one would expect extreme points to play a role here.
The converse of Proposition \ref{P9.1} is false, however. A counterexample is provided by
any ergodic model whose Lyapunov exponent satisfies $\gamma\ge c>0$. This behavior has been
established for the Lyapunov exponent of the Almost Mathieu operator for large coupling
\cite{BouJit} (in fact, Bourgain-Jitomirskaya compute the Lyapunov exponent exactly).
Such a Lyapunov exponent is not an extreme point of $\mathcal L$, for the simple reason
that $\gamma\pm c\in\mathcal L$ also, and of course $\gamma=\frac{1}{2}(\gamma+c + \gamma-c)$.
\begin{proof}
Suppose that $\Gamma\in\mathcal L$ is an extreme point, and let $\mu\in\mathcal M_0$
be an invariant measure so that $\Gamma=\gamma_{\mu}$. We now use Choquet theory
(see \cite{Phelps}, especially Sections 3 and 12 of this reference)
to decompose $\mu=\int\nu\, d\sigma(\nu)$ into ergodic measures
$\nu$ on the Borel sets of $\mathcal J_2$. This means that
\[
\int_{\mathcal J_2} f(J)\, d\mu(J) = \int_{\mathcal M} d\sigma(\nu)\int_{\mathcal J_2}
d\nu(J)\, f(J)
\]
for all bounded Borel functions $f$. Choquet's Theorem says that there is such
a measure $d\sigma$, with the following additional properties: it is a probability measure on the
Borel sets of the space $\mathcal M$ of invariant probability measures on (the Borel sets of)
$\mathcal J_2$ (with the topology induced by the weak-$*$ topology of the regular Borel
measures on $\mathcal J_2$, viewed as the dual of $C(\mathcal J_2)$). Moreover, and this
is crucial, $d\sigma$ is supported by the subset of \textit{ergodic }measures.

We claim that we then also have that
\begin{equation}
\label{9.1}
\gamma_{\mu}(z)=\int_{\mathcal M} \gamma_{\nu}(z)\, d\sigma(\nu)
\end{equation}
for $z\in\C^+$. Indeed, if we set
\[
L_n(J) = \max\{ \ln a_0(J), -n \} ,
\]
say, then monotone convergence, applied a total of three times, shows that
\begin{align*}
\ln A_{\mu} & = \int_{\mathcal J_2}\ln a_0(J)\, d\mu(J)
= \lim_{n\to\infty} \int_{\mathcal J_2} L_n(J)\, d\mu(J) \\
& = \lim_{n\to\infty} \int_{\mathcal M} d\sigma(\nu) \int_{\mathcal J_2} d\nu(J)\, L_n(J)\\
& = \int_{\mathcal M} d\sigma(\nu)\, \lim_{n\to\infty} \int_{\mathcal J_2} d\nu(J)\, L_n(J)
= \int_{\mathcal M} \ln A_{\nu}\, d\sigma(\nu) .
\end{align*}
(This also shows that $d\sigma$ is supported by $\mathcal M_0$.)
Furthermore, by just chasing definitions, we can also easily confirm that
$\int f\, dk_{\mu} = \int d\sigma(\nu)\int dk_{\nu}\, f$ for continuous $f$, so we do obtain \eqref{9.1}
by integrating the Thouless formula for $\gamma_{\nu}$ with respect to $d\sigma$.

Now $\gamma_{\mu}$ is an extreme point by assumption, so if $M\subset\mathcal M$ is any
Borel subset, then necessarily $\int_M\gamma_{\nu}\, d\sigma = \sigma(M)\gamma_{\mu}$
also. In particular, sets of the type
\[
M_{z,\epsilon} = \left\{ \nu\in\mathcal M_0 : \gamma_{\nu}(z)\ge \gamma_{\mu}(z)+\epsilon \right\} ,
\]
with $z\in\C^+$, $\epsilon>0$ all satisfy $\sigma(M_{z,\epsilon})=0$, and of course the same
goes for sets defined by an inequality of the form $\gamma_{\nu}(z)\le \gamma_{\mu}(z)-\epsilon$.
Thus, by taking a suitable countable union, we see that $\gamma_{\nu}\equiv\gamma_{\mu}$
for $\sigma$-almost every $\nu\in\mathcal M_0$. As pointed out above, almost all of these
measures $\nu$ are also ergodic.
\end{proof}
So it would be interesting to know what the extreme points of $\mathcal L$ are.
As observed above, $\gamma$ is not an extreme point if $\inf\gamma >0$. At the other end
of the spectrum, we have the following statement, which we adapted
from \cite[Theorem 6.3]{Kot85} and its proof.
\begin{Proposition}
\label{P9.2}
Let $(A,dk)\in\mathcal D$, and let $\gamma\in\mathcal L$ be the corresponding
Lyapunov exponent. Write $E=\textrm{\rm top supp }dk\subset [-2,2]$.
Suppose that one of the following equivalent conditions holds:\\
(a) $A=\textrm{\rm cap }E$, $dk=d\omega_E$;\\
(b) $\gamma(t)=0$ for quasi every $t\in E$;\\
(c) $\gamma(t)=0$ for $\omega_E$-almost every $t\in E$.

Then $\gamma$ is an extreme point of $\mathcal L$.
\end{Proposition}
So here we assume that $\gamma=0$ essentially everywhere where this function can be
equal to zero. Thus there is a huge gap between the Proposition and our first observation
that $\gamma$ is not an extreme point if $\gamma\ge c>0$ everywhere.
\begin{proof}
The equivalence of (a)--(c) follows from a routine application of
potential theoretic tools; compare, for example, \cite{Simpot}. We sketch the argument here
for the reader's convenience. First of all, if (a) is assumed, then what (b) asserts
is known as \textit{Frostman's Theorem }\cite[Theorem 3.3.4]{Ran}. Next, (b) clearly
implies (c) since $\omega_E$ gives zero weight to all sets of capacity zero. If (c) holds,
then we can integrate the Thouless formula with respect to $\omega_E$ and use
Fubini's theorem to obtain that
\[
0 = -\ln A + \int_{[-2,2]}\Phi_{\omega_E}(t)\, dk(t) =
\ln (\textrm{cap }E/A) .
\]
The last step again depends on Frostman's Theorem. So we indeed have
that $A=\textrm{cap }E$. On the other hand, we may also
integrate with respect to $dk$, and we then obtain that
\[
I(dk)\equiv \int_{[-2,2]}dk(t)\int_{[-2,2]}dk(x)\, \ln |t-x| \ge \ln A .
\]
The equilibrium measure $\omega_E$ may be characterized as the measure
that maximizes $I$ among all probability measures supported by $E$,
and this maximum value equals $I(\omega_E)=\ln\textrm{cap }E$. Thus it now follows
that $dk=d\omega_E$, and we have obtained (a).

Such a $\gamma$ clearly is an extreme point. Indeed, if $\gamma=\frac{1}{2}(\gamma_1 +\gamma_2)$,
then, by Theorem \ref{Tdata}, we must also have that $dk=\frac{1}{2}(dk_1 + dk_2)$, so,
in particular, $E_1, E_2\subset E$ and hence $\gamma_j=0$ quasi everywhere on $E_j$ also.
As we just saw, this property
identifies $dk_j=d\omega_{E_j}$ as the corresponding equilibrium measures. As $\gamma_j>0$
on $E_j^c$, it in fact follows that $E_1=E_2=E$ and thus $\gamma_1=\gamma_2=\gamma$.
\end{proof}
This provides a class of examples where ergodic measures can always be found. We do not know
if there are any $\Gamma\in\mathcal L$ that do not admit ergodic measures for their representation.
Note also that a certain subclass of the examples discussed in Proposition \ref{P9.2} has
the much stronger property that \textit{every }$\mu\in\mathcal M_0$ with $\Gamma=\gamma_{\mu}$
is ergodic (which also means that there is only one such $\mu$ because otherwise we could
take convex combinations to obtain non-ergodic $\mu$'s).
This happens when $E$ is a finite gap set with rationally independent gap
labels (this is classical and follows from an analysis of the shift on these spaces;
see \cite[Chapter 9]{Teschl}), but also for certain sets $E$ with infinitely many gaps
and this property (we know this thanks to work of Sodin-Yuditskii \cite{SodYud}).
It is not clear if there are other examples of Lyapunov exponents
$\Gamma$ with this property that there is only one (ergodic)
$\mu$ with $\Gamma=\gamma_{\mu}$.
\section{Invariance under Toda maps}
In this final section, we show that $w$ is invariant under maps of Toda type.
We will give a simple abstract version of this result, which, at the same time,
will also be more general. It is not necessary
here to be familiar with the theory of Toda flows. We consider continuous maps
$\varphi:\mathcal J_2\to\mathcal J_2$ that preserve the shift dynamics in the
sense that $S\varphi=\varphi S$. This also makes sure that the induced map
$\mu\mapsto\varphi\mu$
on the probability measures on the Borel sets of $\mathcal J_2$ preserves the property of
being an invariant measure; recall in this context that the image measure $\varphi\mu$ is defined
by the condition that $\int f\, d(\varphi\mu)=\int f\circ\varphi\, d\mu$ for
continuous functions $f$.
The invariance of $\varphi\mu$ is most elegantly established by observing that a measure is
$S$ invariant precisely if it coincides with its image measure under $S$.
Now the fact that $S$ and $\varphi$ commute implies that
similarly $S\varphi\mu=\varphi S\mu$, and this latter measure equals $\varphi\mu$ by
the invariance of $\mu$.
\begin{Theorem}
\label{T10.1}
Suppose that $\varphi:\mathcal J_2\to\mathcal J_2$ is a bijective continuous transformation
that commutes with the shift, $S\varphi=\varphi S$,
and preserves spectra: $\sigma(\varphi(J))=\sigma(J)$.
Then $\varphi\mu\in\mathcal M_0$ for every $\mu\in\mathcal M_0$ and $w_{\varphi\mu}=w_{\mu}$.
\end{Theorem}
As alluded to above, the time one map of any Toda flow has these properties;
please see \cite[Chapter 12]{Teschl} and \cite{RemToda} for background. In particular,
there is a reasonably large supply of such maps. Note, however, that while (classical)
Toda flows act by unitary conjugation, we are \textit{not }assuming here
that $\varphi(J)$ and $J$ are unitarily equivalent; the spectra are only conserved as sets.

The invariance of $w$ under (genuine)
Toda flows was established earlier by Knill \cite{Knill};
see especially Theorem 5.1 of \cite{Knill}.
\begin{proof}
First of all, it suffices to prove this for ergodic measures $\mu\in\mathcal M_0$.
To see this, we again use the ergodic decomposition of an invariant measure $\mu$ that was
discussed in the previous section (see the proof of Proposition \ref{P9.1}). So write
$\mu=\int\nu\, d\sigma(\nu)$. It follows directly from the definitions that then
similarly $\varphi\mu =\int\varphi\nu\, d\sigma(\nu)$, and the measures $\varphi\nu$,
as well as the measures $\nu$ themselves, are ergodic $\sigma$-almost everywhere.
So if we can show the invariance
of $w_{\nu}$ under $\varphi$
for ergodic $\nu$, then we will obtain the general case from \eqref{9.1}.
A similar argument is possible concerning the claim that $\varphi\mu\in\mathcal M_0$.

Given an ergodic $\mu\in\mathcal M_0$, we will first approximate it by the measures
$\mu_{\epsilon}=F_{\epsilon}\mu$. Here, we use the same notation as in
the proof of Proposition \ref{P2.1}; see \eqref{defFe} and the discussion that follows.
Note that the $\mu_{\epsilon}$ are also ergodic, and recall that
$\gamma_{\mu_{\epsilon}}(z)\to\gamma_{\mu}(z)$ as $\epsilon\to 0+$, for $z\in\C^+$.

We will then approximate these measures $\mu_{\epsilon}$ by periodic measures;
here, we call a probability measure $\rho$ on $\mathcal J_2$ \textit{periodic} if it is
of the form
\[
\rho = \frac{1}{p} \sum_{j=1}^p \delta_{S^j J}
\]
for some $J\in\mathcal J_2$ with $S^p J = J$. We formulate this step as a separate Lemma.
\begin{Lemma}
\label{L10.1}
Let $\mu$ be an ergodic measure on $\mathcal J_2$. Then there are periodic measures
$\rho_n$ so that $\rho_n\to\mu$ in weak-$*$ sense. Moreover, if $\mu$ is supported by
$\mathcal J_2^{(\epsilon)}$, then the $\rho_n$ can be chosen so that they also have this property.
\end{Lemma}
\begin{proof}[Proof of Lemma \ref{L10.1}]
The ergodic theorem says that if $f\in C(\mathcal J_2)$ is given, then
\[
\lim_{p\to\infty} \frac{1}{p} \sum_{j=1}^p f(S^jJ) = \int_{\mathcal J_2} f(J)\, d\mu(J)
\]
for $\mu$-almost every choice of $J$. Since $C(\mathcal J_2)$ is separable, this implies that also
\begin{equation}
\label{0.31}
\mu = \lim_{p\to\infty} \frac{1}{p} \sum_{j=1}^p \delta_{S^jJ}
\end{equation}
in weak-$*$ sense for $\mu$-almost every $J$.
Fix such a $J$, and consider periodic modifications $J'$ of $J$. More precisely,
we let $J'$ have the same coefficients as $J$ on $n=1,2,\ldots, p$ for some
$p\ge 1$, and then continue
periodically. In other words, $(a',b')_n=(a,b)_n$ for $n=1,2,\ldots, p$, and the remaining
coefficients are obtained from the condition that $(a',b')_{n+p}=(a',b')_n$ for all $n\in\Z$.

Recall now how the metric $d$ on $\mathcal J_2$ was defined; see \eqref{defd}.
Since $a,b\in\ell^{\infty}$, it follows that we can
find a constant $C$ so that
\[
d(S^jJ, S^j J') \le C2^{-p^{1/2}} \quad \textrm{for }p^{1/2} \le j \le p-p^{1/2} .
\]
Indeed, if $j$ is from this range, then the coefficients of $S^jJ$ and $S^jJ'$ agree
on an interval centered at $0$ of size at least $p^{1/2}$, and the estimate follows
at once from this. If we now replace $J$ with $J'=J'_p$ in \eqref{0.31}, then
the periodic measures obtained in this way will still converge to $\mu$ because
$S^jJ'$ will be uniformly close to $S^jJ$ for the lion's share of the sum,
and because of the factor $1/p$, the remaining $\approx p^{1/2}$ summands cannot make
an appreciable contribution.

This procedure also establishes the final claim because all $J'_p$ will be
in $\mathcal J_2^{(\epsilon)}$ if $J$ was from this subspace.
\end{proof}
Apply this to the measures $\mu_{\epsilon}$. We obtain periodic measures
$\mu_{n,\epsilon}\to\mu_{\epsilon}$. Now for a periodic measure $\rho$, we certainly have
that $\gamma_{\varphi\rho}=\gamma_{\rho}$. This follows because if $\rho=(1/p)\sum\delta_{S^jJ}$, then
$\varphi\rho=(1/p)\sum \delta_{\varphi S^j J}$, but since $\varphi$ commutes with $S$,
this is again a periodic measure, and it is formed with the periodic Jacobi matrix
$\varphi(J)$. For a periodic operator $J$ and the associated measure $\rho$,
the corresponding data are
$(A_{\rho},dk_{\rho})=(\textrm{cap }E, d\omega_E)$, where $E=\sigma(J)$ (compare also our discussion
of finite gap domains in this context). So they only depend on the spectrum of $J$,
but we assumed that $\varphi$ preserves this.

Thus $\gamma_{\varphi\mu_{n,\epsilon}}=\gamma_{\mu_{n,\epsilon}}$. We now send $n\to\infty$.
Since the measures $\mu_{n,\epsilon}$ are all supported by $\mathcal J_2^{(\epsilon)}$,
we can be sure that $\gamma_{\mu_{n,\epsilon}}\to\gamma_{\mu_{\epsilon}}$; compare again
the proof of Proposition \ref{P2.1} for this step. There is no such additional information
available for the measures $\varphi\mu_{n,\epsilon}$, so we will just use Lemmas
\ref{Ldkcont} and \ref{LM0} here. Notice that we do know that
$\varphi\mu_{n,\epsilon}\to\varphi\mu_{\epsilon}$
as $n\to\infty$. It follows that
\[
\gamma_{\mu_{\epsilon}}(z) = \gamma_{\varphi\mu_{\epsilon}}(z) + c_{\epsilon} ,
\]
with $c_{\epsilon}\ge 0$. If we now also take $\epsilon\to 0+$, then, as just explained,
the left-hand
side will converge to $\gamma_{\mu}$. On the right-hand side,
we again refer to Lemmas \ref{Ldkcont} and \ref{LM0}
to conclude that
\begin{equation}
\label{10.1}
\gamma_{\mu}(z)= \gamma_{\varphi\mu}(z) +c \quad\quad (c\ge 0) .
\end{equation}
In particular, we have learnt from this argument that $-\ln A_{\varphi\mu_{\epsilon}}$
stays bounded as $\epsilon\to 0+$, so Lemma \ref{LM0} does make
sure that $\varphi\mu\in\mathcal M_0$, as claimed.

To obtain the full assertion of the Theorem, all that remains to be done is to
let $\varphi\mu$ and $\mu=\varphi^{-1}(\varphi\mu)$ swap roles. So only
$c=0$ is possible in \eqref{10.1}.
Since $\gamma$ and $w$ determine each other, the claim may be phrased
in terms of $w$, which is what we did in the formulation of the theorem.
\end{proof}

\end{document}